%% file: holonomics.tex
\documentclass[letterpaper, english]{amsart}

\title{Holonomic and perverse logarithmic D-modules}
\author{Clemens Koppensteiner}
\address{Institute for Advanced Study, 1 Einstein Dr., Princeton, NJ, 08540, USA}
\email{clemens@ias.edu}
\author{Mattia Talpo}
\address{Department of Mathematics, Imperial College London, London SW7 2AZ, UK}
\email{m.talpo@imperial.ac.uk}

 % correct?

\usepackage[utf8]{inputenc}
\usepackage[T1]{fontenc}
\usepackage{lmodern}
\usepackage{microtype}

\usepackage{amsmath,amsfonts,amssymb}
\usepackage{mathtools,extpfeil,xfrac}
\usepackage{amsthm,thmtools}

\usepackage{enumitem}
\setlist[enumerate,1]{label=(\roman*)}
\setlist[enumerate,2]{label=(\alph*)}
\setlist[enumerate,3]{label=(\Alph*)}
\setlist[enumerate,4]{label=\arabic*.}

\usepackage[
    backend=biber,
    style=alphabetic,
    maxnames=100,
    maxalphanames=100,
]{biblatex}

\DeclareLabelalphaTemplate{
    \labelelement{
        \field[final]{shorthand}
        \field{label}
        \field[varwidthlist,strside=left,names=4]{labelname}
    }
}

\DeclareFieldFormat{labelalpha}{#1}
\DeclareFieldFormat{extraalpha}{#1}

\newcommand{\stackcite}[1]{\cite[\href{http://stacks.math.columbia.edu/tag/#1}{Tag~#1}]{stacks-project}}

\usepackage{mparhack}  % correct placement of margin notes
\usepackage{todonotes}
\usepackage{csquotes}
\usepackage[overload]{textcase}
\usepackage{url}

\usepackage[colorlinks=false,unicode=true,bookmarksdepth=2]{hyperref}
\usepackage{bookmark}

\theoremstyle{plain}
\newtheorem{Thm}{Theorem}[section]

\newtheorem{Prop}[Thm]{Proposition}
\newtheorem{Cor}[Thm]{Corollary}

\newtheorem{Lem}[Thm]{Lemma}

\newtheorem{DefLem}[Thm]{Definition and Lemma}

\theoremstyle{definition}
\newtheorem{Def}[Thm]{Definition}

\newtheorem*{Notation}{Notation}

\theoremstyle{remark}
\newtheorem{Rem}[Thm]{Remark}

\newtheorem{Ex}[Thm]{Example}

\input{macros.tex}

\newcommand\gr{\operatorname{gr}}       % associated graded
\newcommand\tauls{\tau^{\mathrm{p}}}    % truncation functors for the log perverse t-structure
\newcommand\taustand{\tau}              % truncation functors for the standard t-structure

\newcommand\lsDqcoh{{}^{\mathrm{p}}\cat D_{\mathrm{qc}}}
\newcommand\lsH{{}^{\mathrm p}\mkern-2.5muH}

\newcommand{\cD}{\mathcal{D}}
\newcommand{\cO}{\mathcal{O}}
\newcommand{\cT}{\mathcal{T}}
\newcommand{\cK}{\mathcal{K}}

\newcommand{\bA}{\mathbb{A}}

\newcommand{\bD}{\mathbb{D}}
\newcommand{\bN}{\mathbb{N}}
\newcommand{\bZ}{\mathbb{Z}}

\newcommand{\TV}[1]{\mathsf{A}_{#1}}
\newcommand{\ITV}[2]{\mathsf{A}_{#1,#2}}
\newcommand{\gp}{\mathrm{gp}}

\newcommand\RomanNum[1]{#1} % looks better in references with the current font.

\begin{document}

\begin{abstract}
    We introduce the notion of a holonomic D-module on a smooth (idealized) logarithmic scheme and show that Verdier duality can be extended to this context.
    In contrast to the classical case, the pushforward of a holonomic module along an open immersion is in general not holonomic.
    We introduce a \enquote{perverse} t-structure on the category of coherent logarithmic D-modules which makes the dualizing functor t-exact on holonomic modules.
    Conversely this t-exactness characterizes holonomic modules among all coherent logarithmic D-modules.
    We also introduce logarithmic versions of the Gabber and Kashiwara--Malgrange filtrations.
\end{abstract}

\maketitle

\setcounter{tocdepth}{1}
\tableofcontents

\section{Introduction}

Logarithmic geometry \cite{Kato, Handbook} is an enhanced version of algebraic geometry, where on top of the usual scheme structure one also has an additional sheaf of monoids, that keeps track of additional data of interest. A prototypical example is a smooth variety $X$ equipped with a simple normal crossings divisor $D$, that is thought of as a ``boundary'' in $X$. In this case the log structure on $X$ is essentially equivalent to the datum of $D$. More exotic examples are given by fibers of a morphism of pairs $(X,D)\to (Y,E)$ of this kind.

One of the main reasons for working with log schemes is that sometimes a variety is singular in the classical sense, but can be equipped with a log structure that makes it ``smooth'' in the logarithmic category, for an appropriate meaning of the term. With smoothness comes an array of useful tools that are otherwise unavailable. For example, every toric variety has a canonical log structure, and the resulting log scheme is smooth in the logarithmic sense. If one allows the base field to acquire a non-trivial log structure as well (or if one considers idealized log structures), then even more varieties become smooth in this generalized sense.

One instance of this phenomenon is that a smooth log scheme always has a locally free sheaf of differentials.
Thus the language of logarithmic geometric lends itself very well to the study of bundles with a logarithmic connection.
For the purpose of this introduction, let us fix a smooth complex log scheme $X$.
For example, $X$ might be given by a smooth complex variety together with a boundary given by a simple normal crossings divisor $D$.
In this setting a log integrable system on $X$ is a coherent sheaf with a connection which potentially has logarithmic singularities along $D$.
Kato--Nakayama \cite{KatoNakayama:1999:LogBettiLogEtaleLogDeRham,IllusieKatoNakayama:2005:QuasiunipotentLogRiemannHilbert} and Ogus \cite{Ogus:2003:LogarithmicRiemannHilbertCorrespondence} extended the classical Riemann--Hilbert correspondence to log integrable systems.
In essence, they establish an equivalence between the category of log integrable systems and \enquote{locally constant sheaves} on the so-called Kato--Nakayama space $X^{\log}$ attached to $X$.

The category of coherent sheaves with logarithmic connection can naturally be extended to the category of logarithmic D-modules.
The basic formalism for logarithmic D-modules is entirely parallel to the one for D-modules in classical geometry, so it is natural to ask whether the Riemann--Hilbert correspondence can be extended to an equivalence between some category of \enquote{regular holonomic} logarithmic D-modules and some notion of \enquote{perverse} sheaves on $X^{\log}$.

\subsection*{Holonomic D-modules in logarithmic geometry}

In this paper, as a first step towards this conjectural Riemann--Hilbert correspondence, we define and investigate a category of holonomic logarithmic D-modules.
Before we go into details let us show with an example why this is not entirely trivial.

Let $X$ be the log line, i.e.~the scheme $\as[\CC] 1$ with the logarithmic structure given by the divisor $\{0\}$.
In this case the sheaf $\lD_X$ of logarithmic differential operators is the subsheaf of the usual sheaf of differential operators generated by the coordinate $x$ and the log vector field $x\frac{\partial}{\partial x}$.
Thus for example the structure sheaf and the skyscraper $\CC$ at the origin (with action $x\frac{\partial}{\partial x} \cdot 1 = 0$) are examples of logarithmic D-modules on the log line.

The characteristic variety of such modules can be defined as in the classical case.
It is a conical subvariety of the logarithmic cotangent bundle.
Since both the structure sheaf and the skyscraper at the origin are coherent $\O_X$-modules, their characteristic varieties are contained in the zero section.
Hence their dimensions are one and zero respectively.
We see from this example that we cannot expect the classical Bernstein inequality (stating that the dimension of the characteristic variety is always at least equal to the dimension of $X$) to hold in the logarithmic context.
On the other hand, since both of these sheaves are log integrable systems, they should certainly be considered to be \enquote{holonomic}.

Geometrically, we can see the failure of the Bernstein inequality as the fact that the logarithmic cotangent bundle only has a canonical Poisson structure, rather than a canonical symplectic structure.

In the simple example above one can define Verdier duality in the usual way as (shifted) derived $\sheafHom$ to $\lD_X$, twisted by a line bundle.
A simple computation (see Example~\ref{ex:duality_fails}) shows that the dual of $\O_X$ is concentrated in cohomological degree $0$, while the dual of the skyscraper is concentrated in degree $1$.
This is in contrast with the fundamental fact in the classical theory that holonomic D-modules can be characterized as those coherent D-modules for which the dual is concentrated in degree $0$.
This statement ultimately underlies the existence of most of the six functor formalism for classical holonomic D-modules.

Finally, again for the case of the log line $X=\as 1$, consider the open embedding $j\colon U = X \setminus \{0\} \hookrightarrow X$.
The structure sheaf $\O_U$ is a holonomic $\lD_U$-module, but the pushforward $j_*\O_U = \CC[x,x^{-1}]$ cannot be a holonomic logarithmic $\lD_X$-module.
Indeed, it is not even coherent.
Thus one cannot hope that the pushforward always preserves holonomicity.
Related to this is the fact that in the logarithmic context holonomic D-modules are generally no longer artinian.
For example the holonomic $\lD_X$-module $\O_X$ has the infinite descending chain of submodules $\O_X \supseteq x\O_X \supseteq x^2\O_X \supseteq \cdots$.

In this paper we provide solutions to these problems.
Firstly we note that the issue with the skyscraper at the origin in the above example is that the support is too small when measured classically.
However, the support is entirely contained in a log stratum.
Hence we should view the characteristic variety as an \emph{idealized} log variety and measure its dimension as such.
With this we can prove the following logarithmic version of the Bernstein inequality (Theorem~\ref{thm:log_bernstein}):
\[
    \logdim\Ch(\sheaf F) \ge \logdim X.
\]
In other words, the Bernstein inequality holds if we use a log geometric version of the dimension.
We call a logarithmic D-module \emph{holonomic} if this is an equality.

For the example of the log line $\as 1$, any subset of the fiber of the log cotangent bundle over the origin has log dimension one more than classical dimension.
Thus the characteristic variety of the skyscraper $\CC$ at the origin, which is just the origin itself, has log dimension $1$.
So this sheaf is holonomic.
On the other hand the characteristic variety of the skyscraper $\CC[x\frac{\partial}{\partial x}]$ is the whole fiber over the origin of the cotangent bundle and thus has log dimension $2$.
So in contrast to the classical theory the skyscraper $\CC[x\frac{\partial}{\partial x}]$ is not holonomic.
Indeed this meshes well with expectations on the constructible side via the conjectural Riemann--Hilbert correspondence.

Since we only require the spaces in this paper to be logarithmically smooth (and hence the underlying classical scheme is allowed to have singularities), the definition of duality needs some care.
Fortunately, we can adapt the existing theory of rigid dualizing complexes from non-commutative geometry.
This does however leave open the problem of duality not restricting to an exact autoequivalence of the holonomic abelian subcategory. A major part of this paper is devoted to a solution of this issue.

Inspired by the work of Kashiwara~\cite{Kashiwara:2004:tStructureOnHolonomicDModuleCoherentOModules} we define a kind of perverse t-structure on the derived category of logarithmic D-modules, which we call the \emph{log perverse t-structure}.
(We will recall the definition of a t-structure in Section~\ref{sec:t-structure}.)
The upshot is that the duality functor is exact when viewed as a functor from holonomic modules to the heart of the log perverse t-structure and vice-versa, that this characterizes the subcategory of holonomic modules.
For example this can be used to generalize the classical proof that proper pushforward preserves holonomicity given in \cite[Section~5.4]{mebkhout}.

This still leaves the unfortunate fact that the pushforward of a holonomic logarithmic D-modules does not need to be holonomic -- or even coherent -- in general.
In some cases, such as the example of $\CC[x,x^{-1}]$ mentioned above, this can be solved by introducing a generalization of the Kashiwara--Malgrange V-filtration, a topic that we will treat in detail in future work.
However, as Remark~\ref{rem:no_extension} shows, sometimes it is impossible to find any extension of a holonomic module from an open subvariety.
This is mainly a problem when one wants to reduce a statement to the compact case.
The upshot is that one has to be quite careful when compactifying, cf.~Remark~\ref{rem:good_compactification}. 
%This is achieved by comparing the \enquote{logarithmic Gabber filtration} by log dimension of the characteristic variety of a submodule with the \enquote{logarithmic Sato--Kashiwara filtration}, which is defined via duality and truncation functors of the log perverse t-structure.

Finally, let us mention that since the category of holonomic logarithmic D-modules is not self-dual under Verdier duality, neither is the expected category of \enquote{perverse sheaves} on the constructible side of the logarithmic Riemann--Hilbert correspondence.
In fact the existence of the log perverse t-structure implies that there exist two categories of \enquote{perverse sheaves}, corresponding to holonomic D-modules either via the de-Rham or the solution functor.
This might explain why these categories so far have not appeared in the literature.
We will discuss the constructible side in a separate upcoming paper.

\subsection*{Outline of the paper}

Let us briefly outline the content of each section of this paper.
In Section~\ref{sec:prelims} we collect and review existing results from logarithmic geometry and the theory of non-commutative dualizing complexes.
This also serves to fix notation.
In Section~\ref{sec:log_D_modules} we define and examine the various categories of logarithmic D-modules.
In particular, Section~\ref{sec:Dmods-and-strata} contains many of the main lemmas of how logarithmic D-modules behave with respect to restriction and local cohomology along log strata.
The main result is a vanishing theorem for local cohomology along such strata.
Finally, in Section~\ref{sec:holonomics}, we show the logarithmic Bernstein inequality and define the category of holonomic logarithmic D-modules.

The rest of the paper concerns the pay-off of this work.
In Section~\ref{sec:t-structure} we define the log perverse t-structure and study its interaction with duality.
In Section~\ref{sec:filtrations}, we introduce logarithmic analogues of the Gabber and Sato filtrations of holonomic D-modules.

\subsection*{Future plans}

As noted above, our main goal is to extend the Riemann--Hilbert correspondence for logarithmic connections of Kato--Nakayama and Ogus to logarithmic D-modules.
Thus in future work we will define categories of perverse sheaves on the Kato--Nakayama space and investigate the interaction of the \enquote{enhanced} de Rham functor of \cite{Ogus:2003:LogarithmicRiemannHilbertCorrespondence} with the constructions in the present paper on the one hand and the yet-to-be-defined category of constructible sheaves on the Kato--Nakayama space.
Of particular importance will be a generalized version of the Malgrange--Kashiwara V-filtration, which corresponds to the \enquote{$\Lambda$-filtration} on the constructible side.
This will also yield further insight into the behaviour of nearby cycles in the logarithmic setting, generalizing the results of \cite{IllusieKatoNakayama:2005:QuasiunipotentLogRiemannHilbert}.
Moreover, we intend to investigate how root stacks \cite{borne-vistoli, talpo-vistoli} and parabolic sheaves interact with this picture, especially in view of their relationship with the Kato--Nakayama space \cite{knvsrootI, knvsrootII, knparabolic}.

We are also interested in applications of the theory of logarithmic D-modules to geometric representation theory (where the classical theory of D-modules already plays a very important role).
In particular, we plan to use the theory in the investigations of the first-named author on the cohomological support theory of D-modules \cite{Koppensteiner:arXiv:HochschildCohomologyOfTorusEquivariantDModules}.
Recently Ben-Zvi, Ganev and Nadler reformulated the celebrated Beilinson--Bernstein localization theorem in terms of filtered D-modules on the wonderful compactification \cite{BGN:BB-localization}.
This construction might be best viewed via the logarithmic structure on the wonderful compactification and it will be interesting to see its interaction with a logarithmic Riemann--Hilbert correspondence.
Finally, we expect that the constructions of this paper have analogues in positive characteristic and arithmetic situations.

\subsection*{Acknowledgements}

The work of C.K. was partially supported by the National Science Foundation under Grant No.~1638352, as well as the Giorgio and Elena Petronio Fellowship Fund. M.T. was partially supported by a PIMS postdoctoral fellowship. We are grateful to Nathan Ilten for useful exchanges, and to the anonymous referee for several helpful comments.

\subsection*{Notations and assumptions}

Unless stated otherwise, throughout this note $X$ will denote a separated, smooth (in the logarithmic sense) idealized log scheme, of finite type over an algebraically closed base field $\kk$ of characteristic $0$.
Such an $X$ comes with a stratification in locally closed subschemes. We will write $X^k$ for the closure in $X$ of the union of all strata of codimension $k$, and call this subset a ``logarithmic stratum''.
Thus we have a descending sequence of closed subschemes
\[
    X = X^0 \supseteq X^1 \supseteq \cdots \supseteq X^{\dim X} \supseteq X^{\dim X+1} = \emptyset.
\]
For convenience we set $X^k = X$ for $k < 0$ and $X^k = \emptyset$ for $k > \dim X$.
We usually endow each $X^k$ with the induced idealized log structure, making it a smooth idealized log scheme.
We will write $TX$ and $T^*X$ for the logarithmic tangent and cotangent bundle respectively.
Similarly, $\Omega_X$ will be the sheaf of logarithmic differentials and $\canon[X]$ the logarithmic canonical bundle.

Unless mentioned otherwise all functors between categories of sheaves are derived, even though we drop the markers $\RR$ and $\LL$.
In particular, $\sheafHom$ is the (right) derived internal Hom and $\otimes$ the (left) derived tensor product.

Unless mentioned otherwise, the term ``$\lD_X$-module'' will always mean a logarithmic D-module with respect to the log structure on $X$.

\section{Preliminaries}\label{sec:prelims}

In this section we gather for the convenience of the reader the basic definitions and some facts about log geometry and coherent duality that are required for the rest of the paper. 

\subsection{Smooth (idealized) logarithmic schemes}

We refer to \cite{Kato,Handbook,Ogus:logbook} for more detailed treatments of logarithmic geometry. On top of logarithmic schemes, we will often have to use \emph{idealized} logarithmic schemes, a variant of the notion, first introduced by Ogus (to the best of our knowledge).

All monoids in this paper will be commutative, and the operation will be written additively. They will also always be fine and saturated (see \cite[Section~\RomanNum{I}.1.3]{Ogus:logbook} for the definitions). Given a monoid $P$, we will denote by $\kk[P]$ the monoid algebra of $P$, i.e. the $\kk$-algebra generated by variables $t^p$ with $p\in P$, with relations $t^0=1$ and $t^{p+p'}=t^p\cdot t^{p'}$ for every $p,p'\in P$.

\begin{Def}\label{def:smooth log scheme}
A \emph{smooth log scheme} is a pair $(X,D)$ consisting of a scheme $X$ equipped with a closed subscheme $D\subseteq X$, with the following property: around every point $x\in X$, the pair $(X,D)$ is \'etale locally isomorphic to an affine normal toric variety $\TV{P}=\Spec \mathbb{C}[P]$, equipped with its reduced toric boundary $\Delta_P$.
\end{Def}

Note that every such $X$ has to be smooth away from $D$, but can be singular along the boundary. Another name used in the literature for the open immersion $X\setminus D\subseteq X$ is \emph{toroidal embedding} \cite{toroidal}.

If $(X,D)$ is a smooth log scheme, the sheaf $M=\{f\in \cO_X\mid f|_{X\setminus D}\in \cO_{X\setminus D}^\times\}$ is a sheaf of (multiplicative) submonoids of $\cO_X$. Together with the inclusion $M\subseteq \cO_X$, this sheaf of monoids gives a \emph{log structure} on the scheme $X$, essentially equivalent to the datum of $D$.

\begin{Def}
A \emph{log scheme} is a triple $(X,M,\alpha)$, where $X$ is a scheme, $M$ is a sheaf of monoids on the small \'etale site of $X$, and $\alpha\colon M\to \cO_X$ is a homomorphism of monoids, where $\cO_X$ is equipped with multiplication of sections, such that the induced morphism $\alpha|_{\alpha^{-1}\cO_X^\times}\colon \alpha^{-1}\cO_X^\times\to \cO_X^\times$ is an isomorphism.
\end{Def}

Every smooth log scheme is in particular a log scheme, but log schemes are more general. In particular a log structure $(M,\alpha)$ can be pulled back along arbitrary morphisms of schemes, something that cannot be done with pairs $(X,D)$: given a log scheme $(X,M,\alpha)$ and a morphism of schemes $f\colon Y\to X$, one can pull back $M$ and $\alpha$ to $Y$, obtaining a homomorphism of sheaves of monoids
\[
  f^{-1}\alpha\colon f^{-1}M\to f^{-1}\cO_X\to \cO_Y,
\]
but this is not a log structure in general. It can be made into one in a universal way, by considering the pushout $f^*M=f^{-1}M\oplus_{(f^{-1}\alpha)^{-1}\cO_Y^\times} \cO_Y^\times$ with the induced homomorphism to $\cO_Y$. In particular, the log structure associated with a pair $(X,D)$ as above can be restricted to components of $D$.

The quotient sheaf $\overline{M}=M/\cO_X^\times$ is called the \emph{characteristic sheaf} of the log scheme $(X,M,\alpha)$. If the stalks of $\overline{M}$ are fine and saturated (in particular this is true if the log scheme is smooth), then the stalks of the associated sheaf of groups $\overline{M}^\gp$ are torsion-free abelian groups of finite rank.

A morphism of log schemes $(X,M,\alpha)\to (Y,N,\beta)$ is a morphism of schemes $f\colon X\to Y$, together with a homomorphism of monoids $f^\flat\colon f^{-1}N\to M$, such that the composite $f^{-1}N\to M\to \cO_X$ coincides with the composite $f^{-1}N\to f^{-1}\cO_Y\to \cO_X$. Such a morphism is called \emph{strict} if $f^\flat$ induces an isomorphism $f^{-1}\overline{N}\xrightarrow{\cong} \overline{M}$.

The category of log schemes contains the category of schemes as a full subcategory: every scheme $X$ has a \emph{trivial log structure}, given by the sheaf of monoids $M=\cO_X^\times$ with the inclusion $j\colon \cO_X^\times \hookrightarrow\cO_X$. Every log scheme $(X,M,\alpha)$ admits a canonical map to the corresponding trivial log scheme $(X,\cO_X^\times,j)$. The functor that equips a scheme with the trivial log structure is right adjoint to the forgetful functor from log schemes to schemes. 

We also recall the concept of a \emph{chart}: assume that $X$ is a scheme, $P$ is a monoid and $\phi\colon P\to \cO(X)$ is a homomorphism of monoids. Then we obtain a homomorphism of sheaves of monoids $\underline{P}\to \cO_X$ (where $\underline{P}$ denotes the constant sheaf with stalks $P$), and by applying the same construction outlined above for the pullback log structure, this induces a log structure $M_\phi\to \cO_X$ on $X$. If $(X,M, \alpha)$ is a log scheme, a chart is a homomorphism of monoids $P\to M(X)$, giving $\phi\colon P\to M(X)\xrightarrow{\alpha(X)} \cO(X)$ such that the induced morphism of log structures $M_\phi\to M$ is an isomorphism. For every fine monoid $P$, the log structure on $\TV{P}$ induced as described above by the toric boundary $\Delta_P$ has a chart, given by the tautological homomorphism of monoids $P\to\kk[P]$.

A chart for $(X,M,\alpha)$ is the same thing as a strict morphism $X\to \TV{P}$ for some fine monoid $P$. The existence of local models in Definition \ref{def:smooth log scheme} is precisely the requirement that \'etale locally on $X$ there are charts $X\to \TV{P}$, that are moreover \'etale.

Every smooth log scheme has a stratification $\{S^k\}_{k\in \bN}$, given by the locally closed subsets $S^k\subseteq X$ where the sheaf of free abelian groups $\overline{M}^\gp$ has rank $k$.
We will denote by $X^k$ the Zariski closure of $S^k$ in $X$. This is the subset of $X$ of points where the rank of $\overline{M}^\gp$ is \emph{at least} $k$. We will refer to the $X^k$ as \emph{logarithmic strata} of $X$. There is a descending chain of inclusions
\[
    X = X^0 \supseteq X^1 \supseteq \cdots \supseteq X^{\dim X} \supseteq X^{\dim X+1} = \emptyset.
\]
Unless otherwise mentioned, these closed subsets will always be equipped with the reduced subscheme structure. The subscheme $X^k$ (if it is non-empty) has codimension $k$ in $X$, and for every $k$ the scheme $X^k\setminus X^{k+1}$ is classically smooth (because this is true for the orbit stratification of toric varieties).

\begin{Def}
    Let $Z$ be a closed subset of a smooth log scheme $X$.
    Then the \emph{log dimension} of $Z$ is
    \[
        \logdim Z =  \max_k \bigl(\dim ( Z \cap X^k) + k\bigr).
    \]
\end{Def}

In particular note that the log dimension of $X$ itself coincides with $\dim X$.

As we mentioned, in order to work on the log strata we will need to consider more general log schemes, \'etale locally modelled on the boundary of an affine toric variety rather than on the whole variety. These objects are also smooth, but in an even laxer sense than smooth log schemes. Recall that an \emph{ideal} of a monoid is a subset $K\subseteq P$, such that $i+p\in K$ for every $i\in K$ and $p\in P$.

\begin{Def}
An \emph{idealized} log scheme is a pair consisting of a log scheme $(X,M,\alpha)$, and a sheaf of ideals $I\subseteq M$ such that $\alpha(I)\subseteq \{0\}\subseteq \cO_X$.
\end{Def}

Every log scheme can be made into an idealized log scheme, via the empty sheaf of ideals $I=\emptyset\subseteq M$. For every idealized log scheme $(X,M,\alpha,I)$ there is a canonical morphism $(X,M,\alpha, I)\to (X,M,\alpha,\emptyset)$, and the resulting functor from log schemes to idealized log schemes is the right adjoint of the forgetful functor from idealized log schemes to log schemes.

For a fine saturated monoid $P$ with an ideal $K$, set $\ITV{P}{K}=\Spec \kk[P]/\langle K \rangle$, where $\langle K\rangle$ denotes the ideal of $\Spec \kk[P]$ generated by elements of the form $t^i$ with $i\in K$. This is an idealized log scheme: if $\alpha_{P,K}\colon M_{P,K}\to \cO_{\ITV{P}{K}}$ denotes the tautological log structure on $\ITV{P}{K}\subseteq \TV{P}$, then the sheaf of ideals of $M_{P,K}$ generated by the image in $M_{P,K}$ of the subsheaf $\underline{K}\subseteq \underline{P}$ via the chart morphism $\underline{P}\to M_{P,K}$ is sent to $0$ by $\alpha_{P,K}$. More generally, a \emph{chart} for an idealized log scheme $(X,M,\alpha, I)$ is a strict morphism $X\to \ITV{P}{K}$, where $P$ is a fine saturated monoid and $K\subseteq P$ is an ideal, such that the image of the induced morphism $\underline{K}\to M$ generates the sheaf of ideals $I\subseteq M$. 

If $(X,M,\alpha)$ is a smooth log scheme, the closed subsets $j\colon X^k\hookrightarrow X$ are naturally idealized log schemes: the log structure is the pullback of the log structure of the ambient, and the sheaf of ideals $I\subseteq j^*M$ is the subsheaf $(j^*\alpha^{-1})\{0\}$.

In analogy with Definition \ref{def:smooth log scheme}, we define smoothness for idealized log schemes by asking that they be locally modelled on idealized log schemes of the form $\ITV{P}{K}$.

\begin{Def}
An idealized log scheme $(X,M,\alpha, I)$ is \emph{smooth} if \'etale locally around every point $x\in X$ there is a strict morphism $X\to \ITV{P}{K}$ for some fine saturated monoid $P$ and ideal $K\subseteq P$, which is moreover \'etale.
\end{Def}

Note that by definition every idealized log scheme of the form $\ITV{P}{K}$ is smooth. In particular, if $\{t^{p_i}=0\}$ is a set of monomial equations for the boundary $\Delta_P$ with the reduced structure, then $\ITV{P}{\langle p_i\rangle}$ is a smooth idealized log scheme whose underlying scheme is $\Delta_P$. At the other extreme, if $P$ is sharp, for the ``maximal ideal'' $K=P\setminus \{0\}$, the underlying scheme of the smooth idealized log scheme $\ITV{P}{P\setminus \{0\}}$ is the torus-fixed point of the toric variety $\TV{P}$ (with non-trivial log structure, with monoid $M=P\oplus \kk^\times$). For $P=\bN$, the log scheme $\ITV{\bN}{\bN\setminus \{0\}}$ (i.e. the origin of $\bA^1$ equipped with the pullback log structure) is usually called the \emph{standard log point}.

The closed subschemes $X^k$ of a smooth log scheme, equipped with the idealized log structure described above, are smooth idealized log schemes. Conversely, every smooth idealized log scheme can be \'etale locally embedded as a closed subscheme of a smooth log scheme (basically by definition).

\begin{Rem}
There is a general notion of smoothness (usually called ``log smoothness'') of morphisms of (idealized) log schemes. The requirements in the above definitions are equivalent to asking that the structure morphism from the (idealized) log scheme to $\Spec\kk$ equipped with the trivial log (and idealized) structure is smooth. We opted to give ad hoc definitions in this paper, because we do not need the ``full'' machinery.
\end{Rem}

Smooth idealized log schemes also admit stratifications. From now on, let us assume that all smooth idealized log schemes are also \emph{connected}. In this case, the stalks of the sheaf $\overline{M}^\gp$ have the same rank $r$ on a dense open subset of $X$. We call this number the \emph{generic rank} of $(X,M,\alpha,I)$.

If we define $S^k$ and $X^k$ as the subsets of $X$ where the rank of $\overline{M}^\gp$ is exactly (resp. at least) $r+k$, we have a  descending chain of closed subsets
\[
    X = X^0 \supseteq X^{1} \supseteq \cdots \supseteq X^{\dim X} \supseteq X^{\dim X+1} = \emptyset.
\]
As in the non-idealized case, these closed subsets will be always equipped with the reduced subscheme structure.

\begin{Def}
    Let $Z$ be a closed subset of a smooth idealized log scheme $X$, with generic rank $r$.
    Then the \emph{log dimension} of $Z$ is
    \[
        \logdim Z =  \max_k \bigl(\dim ( Z \cap X^k) + k\bigr) + r.
    \]
\end{Def}

In particular, the log dimension of $X$ itself is $\dim X + r$.

\begin{Notation}
From now on, in order to lighten the notation we will denote a (possibly idealized) log scheme $(X,M,\alpha, I)$ just by $X$, and use subscripts for the other symbols (i.e. we will write $M_X,\alpha_X, I_X$) if necessary. Unless there is risk of confusion, we will denote the underlying scheme of the log scheme $X$ again by $X$. If there is such a risk, we will denote the underlying scheme by $\underline{X}$. 
\end{Notation}

\subsection{Logarithmic differentials}

Every log scheme carries a sheaf of logarithmic differentials, constructed in analogy with the sheaf of K\"{a}hler differentials in the non-logarithmic case. This gives rise in particular to a logarithmic cotangent bundle and a sheaf of differential operators, that we use to define and study D-modules, as in the classical theory. 

\begin{Def}\label{def:log derivations}
Let $X$ be a smooth log scheme. The \emph{logarithmic tangent sheaf} $\cT_X$ of $X$ is the subsheaf of the usual tangent sheaf, consisting of the $\kk$-derivations of $\cO_X$ that preserve the ideal defining the closed subscheme $X^1\subseteq X$.
\end{Def}

This agrees with a more general construction, that applies to arbitrary logarithmic schemes \cite[Section~\RomanNum{IV}.1]{Ogus:logbook}. Given an arbitray log scheme $X$, one can define a concept of \emph{logarithmic derivation} with values into an $\cO_X$-module \cite[Definition~\RomanNum{IV}.1.2.1]{Ogus:logbook}. The tangent sheaf $\cT_X$ is then the sheaf of logarithmic derivations with values in $\cO_X$.

 As in the classical case, there is an $\cO_X$-module with a universal logarithmic derivation, that can be explicitly constructed as the quotient
\[
\Omega_X=(\Omega^1_{\underline{X}/\kk}\oplus(\cO_X\otimes_{\bZ} M^\gp))/\cK
\]
where $\Omega^1_{\underline{X}/\kk}$ is the usual sheaf of K\"{a}hler differentials on the scheme $\underline{X}$, and $\cK$ is the submodule generated by sections of the form $(d\alpha(m),0)-(0,\alpha(m)\otimes m)$ for $m\in M$. If $X$ admits local charts (which will be always true for us), this is a coherent $\cO_X$-module, called the \emph{logarithmic cotangent sheaf}. The dual of $\Omega_X$ is the logarithmic tangent sheaf $\cT_X$. See  \cite[Theorem~\RomanNum{IV}.1.2.4]{Ogus:logbook} for details about the construction.

If $X$ is a smooth variety and $D\subseteq X$ is simple normal crossings, $\Omega_{(X,D)}$ coincides with the usual sheaf of logarithmic differential forms on $(X,D)$, i.e. the sheaf of meromorphic differential forms on $X$ with poles of order at most one along components of $D$. This is locally generated by differential forms of the form $dg$ for $g\in \cO_X$  and $d\log f=df/f$ where $f$ is a local equation of $D$.

In the idealized case the sheaf $\Omega_X$ is still defined using the formula above (i.e. by disregarding the idealized structure). In this case though, its dual $\cT_X$ is not a subsheaf of the usual tangent sheaf, because it has sections coming from sections of the sheaf $M$ that map to zero in $\cO_X$. For example, if $X$ is the standard log point, the usual tangent sheaf has rank $0$, while the sheaf of logarithmic derivations $\cT_X$ has rank $1$, and it is generated by the restriction of the derivation $x\frac{\partial}{\partial x}$ from $\bA^1$.

In both cases, the sheaf $\cT_X$ has a Lie bracket, that we will denote by $[-,-]\colon \cT_X\times \cT_X\to \cT_X$ as usual. See \cite[Section~\RomanNum{V}.2.1]{Ogus:logbook} for details. Moreover, as in the classical case, a morphism of (idealized)  log schemes $f\colon X\to Y$ induces a canonical morphism of sheaves $f^*\Omega_Y\to \Omega_X$ (cf.~\cite[Proposition~\RomanNum{IV}.1.2.15]{Ogus:logbook}).

\begin{Ex}\label{example:differentials.toric}
If $X=\TV{P}$ is an affine toric variety, then one can check that $\Omega_X$ is isomorphic to the locally free sheaf $\cO_X\otimes_{\bZ} P^\gp$ (see \cite[Proposition~\RomanNum{IV}.1.1.4]{Ogus:logbook}).

Moreover, for every ideal $K\subseteq P$, for the idealized log scheme $Y=\ITV{P}{K}$ we also have $\Omega_Y=\cO_Y\otimes_{\bZ}P^\gp$. Note that this is still locally free, but if $K$ is non-trivial it has bigger rank than the dimension of $Y$. Moreover, we have $\Omega_Y\cong j^*\Omega_X$ where $j\colon Y\to X$ is the closed embedding.
\end{Ex}

The behaviour of the previous example generalizes in particular to restriction to strata of smooth log schemes. In order to prove this, we reduce to the case of affine toric varieties via the following proposition.

\begin{Prop}[{\cite[Corollary~\RomanNum{IV}.3.2.4]{Ogus:logbook}}]
Let $f\colon X\to Y$ be a strict and classically \'etale morphism of (idealized) log schemes. Then the canonical morphism $f^*\Omega_Y\to \Omega_X$ is an isomorphism.\qed
\end{Prop}

\begin{Prop}\label{prop:differential restriction}
Let $X$ be a smooth idealized log scheme, and $j\colon X^k\hookrightarrow X$ the log stratum of codimension $k$. Then the canonical morphism of $\cO_{X^k}$-modules $j^*\Omega_X\to \Omega_{X^k}$ is an isomorphism.
\end{Prop}

\begin{proof}
The preceding proposition reduces the statement to the case of an affine toric variety $\TV{P}$, which follows from \cite[Corollary~\RomanNum{IV}.2.3.3]{Ogus:logbook}, after noticing that every log stratum is defined by an ideal of the monoid $P$. 
\end{proof}

Note that the isomorphism of the last proposition is a ``logarithmic'' phenomenon, that has no analogue in the classical setting. Moreover, as a consequence also the canonical morphism between logarithmic tangent sheaves $\cT_{X^k}\to j^*\cT_X$ is an isomorphism.

\begin{Prop}
Let $X$ be a smooth idealized log scheme. Then the log cotangent sheaf and the log tangent sheaf are locally free. Their rank coincides with the log dimension of $X$.
\end{Prop}

\begin{proof}
This follows directly from the case of (idealized) toric varieties mentioned in Example \ref{example:differentials.toric}.
\end{proof}

The geometric vector bundles associated with $\cT_X$ and $\Omega_X$ will be denoted by $TX$ and $T^*X$ respectively. We will have no use for the ``ordinary'' tangent and cotangent sheaves (or bundles) of $X$, so there will be no risk of confusion. We will usually denote by $\pi\colon T^*X\to X$ the projection from the log cotangent bundle, and we will always tacitly equip $T^*X$ with the log structure obtained from $X$ by pulling back along $\pi$.

As in the classical case, the wedge powers $\Omega_X^i=\bigwedge^i \Omega_X$ of $\Omega_X$ play an important role. In particular, the top wedge power $\omega_X=\Omega_X^{\logdim X}$ is the \emph{logarithmic canonical bundle} of $X$. Another consequence of Proposition \ref{prop:differential restriction} is that the natural map $j^*\Omega_X^i\to \Omega_{X^k}^i$ for the embedding $j\colon X^k\to X$ of a log stratum is an isomorphism for every $i$ and $k$. In particular the natural map $j^*\omega_X\to \omega_{X^k}$ is an isomorphism (note that $\logdim X^k=\logdim X$ for every $k$).

\subsection{Differential algebras}\label{sec:almost_commutative}

In this section we gather some results about modules over arbitrary sheaves of differential $\O_X$-algebras.
Thus we will temporarily forget the log structure, and until the end of this section $X$ will be a separated noetherian  scheme of finite type over a fixed field $\kk$.
We fix a quasi-coherent sheaf $\sheaf D$ of differential $\O_X$-algebras.
Thus $\sheaf D$ is endowed with an increasing filtration $F$ of $\O_X$-sub-bimodules indexed by $\NN$, compatible with the multiplication, such that the associated graded $\gr \sheaf D$ is a commutative finite type $\O_X$-algebra.
In particular, $\sheaf D$ is a differential quasi-coherent $\O_X$-ring of finite type in the sense of \cite[Definition~5.2]{YekutieliZhang:2006:DualizingComplexesPerverseSheavesOnNoncommutativeRingedSchemes}.
We set $Y = \Spec_{\O_X}\gr\sheaf D$ and let $\pi \colon Y \to  X$ be the projection map.

For a finitely generated $\sheaf D$-module $\sheaf F$ we can as usual (locally) pick a good filtration on $\sheaf F$ and consider its associated graded as an $\O_Y$-module.
The (topological) support of $\gr \sheaf F$ does not depend on the choice of good filtration and will be denoted by $\Ch(\sheaf F)$.
For $\sheaf F \in  \catDbCoh{\sheaf D}$ we set $\Ch(\sheaf F) = \bigcup_{\ell}\Ch(H^\ell(\sheaf F))$.

As usual, given a good filtration of $\sheaf F$, a short exact sequence 
\[
    0 \to \sheaf G \to \sheaf F \to \sheaf H \to 0
\]
of coherent  $\sheaf D$-modules induces good filtrations on $\sheaf G$ and $\sheaf H$ and a short exact sequence of the associated graded modules.
Thus,
\[
    \Ch(\sheaf F) = \Ch(\sheaf G) \cup \Ch(\sheaf H).
\]
In particular, if $\sheaf F'$ is a subquotient of $\sheaf F$, then $\Ch(\sheaf F') \subseteq \Ch(\sheaf F)$.

In this situation Yekutieli and Zhang, based on previous work by van den Bergh \cite{VanDenBergh:1997:ExistenceTheoremsForDualizingComplexes}, introduced the notion of a \emph{rigid dualizing complex} \cite[Definition~3.7]{YekutieliZhang:2006:DualizingComplexesPerverseSheavesOnNoncommutativeRingedSchemes}.
The existence of a rigid dualizing complex is guaranteed by \cite[Theorem~0.3]{YekutieliZhang:2006:DualizingComplexesPerverseSheavesOnNoncommutativeRingedSchemes}.
We fix a choice of such and denote it by $\DC$.
Let us write 
\[
    \DD = \sheafHom_{\sheaf D}({-},\, \DC)\colon \catDbCoh{\sheaf D} \to  \catDbCoh{\sheaf D^{\mathrm{op}}}
\]
for the corresponding dualizing functor, which by definition is an equivalence.
The following proposition is well known under the additional assumption that $Y$ is smooth (cf.~\cite[Section~2]{Bjoerk:1979:RingsOfDifferentialOperators}).

\begin{Prop}\label{prop:nc_dual_vanishing_and_support}
    Let $\sheaf F$ be a coherent $\sheaf D$-module.
    Then,
    \begin{enumerate}
        \item\label{prop:nc_dual_vanishing_and_support:i} $\DD\sheaf F \in  \catDCoh[\ge n]{\sheaf D^{\mathrm{op}}}$ if and only if $n \le - \dim \Ch(\sheaf F)$.
        \item\label{prop:nc_dual_vanishing_and_support:ii} $\dim \Ch\bigl(H^\ell(\DD \sheaf F)\bigr) \le - \ell$ for all $\ell \in  \mathbb{Z}$.
    \end{enumerate}
\end{Prop}

\begin{proof}
    Both statements are local, so that it suffices to prove the corresponding statements for rings.
    By \cite[Theorem~8.1]{YekutieliZhang:2005:DualizingComplexesPerverseModulesOverDifferentialAlgebras}, $\DD\sheaf F \in  \catDCoh[\ge n]{\sheaf D^{\mathrm{op}}}$ if and only if $n$ is less or equal than the negative of the Gelfand--Kirillov dimension of $\sheaf F$.
    By \cite[Corollary~1.4]{ConnellStafford:1989:GKDimAndAssociatedGraded} this is the same as the Gelfand--Kirillov dimension of $\gr \sheaf F$ for some good filtration of $\sheaf F$.
    That in turn is the same as the Krull dimension of $\gr \sheaf F$ \cite[Proposition~7.9]{KrauseLenagan:2000:GrowthOfAlgebrasAndGKDimension}.
    Statement \ref{prop:nc_dual_vanishing_and_support:i} follows.

    Similarly, \ref{prop:nc_dual_vanishing_and_support:ii} is a restatement of \cite[Theorem~0.3]{YekutieliZhang:2005:DualizingComplexesPerverseModulesOverDifferentialAlgebras} under the same equivalences.
\end{proof}

\begin{Cor}\label{cor:nc_dual_vanishing_and_support_derived}
    Let $\sheaf F \in  \catDbCoh{\sheaf D}$.
    Then $\DD\sheaf F \in  \catDCoh[\ge n]{\sheaf D^{\mathrm{op}}}$ if and only if $$\dim \Ch(H^\ell(\sheaf F)) \le - n - \ell$$ for all $\ell \in  \mathbb{Z}$.
\end{Cor}

\begin{proof}
    Let us first assume that $\dim \Ch(H^\ell(\sheaf F)) \le - n - \ell$ for all $\ell \in  \mathbb{Z}$.
    We induct on the cohomological amplitude of $\sheaf F$.
    If $\sheaf F$ is just a shift of a coherent $\sheaf D$-module, then the assertion is just a shift of Proposition~\ref{prop:nc_dual_vanishing_and_support}\ref{prop:nc_dual_vanishing_and_support:i}.
    Otherwise let $p$ be the smallest integer such that $H^p(\sheaf F) \ne 0$.
    Consider the distinguished triangle
    \[
        H^p(\sheaf F)[-p] \to  \sheaf F \to  \taustand_{>p}\sheaf F
    \]
    and dualize to obtain the triangle
    \[
        \DD(\taustand_{>p} \sheaf F) \to 
        \DD\sheaf F \to 
        \DD(H^p(\sheaf F)[-p]).
    \]
    By induction, both of the outer terms lie in cohomological degrees at least $n$.
    Thus so does the middle term.

    For the converse implication let us replace $\sheaf F$ by $\DD\sheaf F$ and $\sheaf D$ by $\sheaf D^{\mathrm{op}}$, so that we assume that $\sheaf F \in \catDCoh[\ge n]{\lD}$ and we have to show that $\dim \Ch(H^\ell(\DD\sheaf F)) \le - n - \ell$ for all $\ell \in  \mathbb{Z}$.
    Again we induct on the cohomological amplitude of $\sheaf F$.
    If $\sheaf F$ is just a shift of a coherent $\sheaf D$-module, then the assertion is just a shift of Proposition~\ref{prop:nc_dual_vanishing_and_support}\ref{prop:nc_dual_vanishing_and_support:ii}.
    Otherwise let $p$ be the smallest integer such that $H^p(\sheaf F) \ne 0$ and consider the distinguished triangle
    \[
        H^p(\sheaf F)[-p] \to  \sheaf F \to  \taustand_{>p}\sheaf F.
    \]
    Applying duality, we get the triangle
    \[
       \DD(\taustand_{>p}\sheaf F)\to   \DD\sheaf F \to  \DD(H^p(\sheaf F)[-p]).
    \]
    The corresponding long exact sequence on cohomology sheaves shows that the cohomology sheaves of $\DD\sheaf F$ fit into short exact sequences with subquotients of the cohomologies of the outer complexes.
    Hence
    \[
        \Ch(H^\ell(\DD\sheaf F)) \subseteq 
        \Ch(H^\ell(\DD(H^p(\sheaf F)[-p]))) \cup 
        \Ch(H^\ell(\DD(\taustand_{>p}\sheaf F))),
    \]
    and the statement follows by induction.
\end{proof}

\begin{Cor}\label{cor:nc_dual_vanishing_and_support_unbounded}
    Let $\sheaf F \in  \catDPlusCoh{\sheaf D}$.
    Then $\sheaf F \in  \catDCoh[\ge n]{\sheaf D}$ if and only if
    \[
        \dim \Ch(H^\ell(\DD\sheaf F)) \le - n - \ell
    \]
    for all $\ell \in  \mathbb{Z}$.
\end{Cor}

\begin{proof}
  For $\sheaf F \in \catDbCoh{\sheaf D}$, this is just the dual of Corollary~\ref{cor:nc_dual_vanishing_and_support_derived}.
  If $\sheaf F$ is only bounded below the statement follows by taking the limit over all truncations $\taustand_{<k}\sheaf F$.
\end{proof}

\begin{Lem}\label{lem:duality_keeps_dim_ch}
    Let $\sheaf F \in  \catDbCoh{\sheaf D}$. Then $\dim \Ch(\DD\sheaf F) = \dim \Ch(\sheaf F)$.
\end{Lem}

\begin{proof}
    By symmetry we only have to show that $\dim \Ch(\DD\sheaf F) \leq  \dim \Ch(\sheaf F)$.
    By a standard induction argument it further suffices to show this assuming that $\sheaf F \in  \catCoh{\sheaf D}$ is a coherent $\sheaf D$-module.
    Set $d = \dim \Ch(\sheaf F)$.
    Then by Proposition~\ref{prop:nc_dual_vanishing_and_support} we immediately see that
    \[
        \dim \Ch(\DD\sheaf F) =
        \max_{\mathclap{-d \leq  \ell \leq  0}} \dim \Ch\bigl(H^\ell(\DD\sheaf F)\bigr) \leq  
        \max_{\mathclap{-d \leq  \ell \leq  0}}-\ell =  
        d.
        \qedhere
    \]
\end{proof}

\section{Logarithmic D-modules}\label{sec:log_D_modules}

We will start by introducing the basic formalism of logarithmic D-modules and exploring some of their properties.
It will turn out that for purposes of induction we will need to consider D-modules supported on a log stratum of $X$.
As explained in the preliminaries, the log strata of $X$ have a natural structure as separated smooth \emph{idealized} log schemes and hence we need to consider D-modules on such.

After introducing the basics, we will discuss how logarithmic D-modules interact with restriction to log strata.
The general idea is that when crossing to a deeper log stratum, coherent $\lD_X$-modules behave very similarly to coherent $\O_X$-modules.

\subsection{Logarithmic D-modules and operations on them}

Let $X$ be a separated smooth idealized log scheme over $\kk$.
The basic set-up of logarithmic D-modules is very similar to that of D-modules on smooth varieties.
Thus we will only give a brief outline and trust that the reader can fill in the details and proofs from any standard reference on D-modules (such as the book \cite{HottaTakeuchiTanisaki:2008:DModulesPerverseSheavesRepresentationTheory}).

\subsubsection{Higher-order logarithmic differential operators}\label{sec:higher.differential}
Let us start by describing higher-order differential operators on smooth (idealized) log schemes. If $X$ is a smooth log scheme, we denote by $\cD_X^{\leq n}$ the subsheaf of the usual sheaf of differential operators of the scheme $X$, generated by the image of the sheaves $\cT^{\otimes i}_X$ for $0\leq i\leq n$. We also set $\cD_X=\bigcup_n \cD_X^{\leq n}$. This is the sheaf of \emph{logarithmic differential operators} on $X$.

In order to write down an explicit local presentation of $\cD_X$ in terms of generators and relations, let us recall a bit of terminology and Lemma 3.3.4 of \cite{AbramovichTemkinWlodarczyk:arXiv:PrincipalizationOfIdealsOnToroidalOrbifolds}, that describes the tangent sheaf $\cT_X$. Let $X$ be a smooth log scheme and $p\in X$ be a point, and denote by $S$ the ``locally closed log stratum'' through $p$. In other words, $S$ is the irreducible component of $X^k\setminus X^{k+1}$ that contains $p$, where $k$ is the rank of the free abelian group $\overline{M}_p^\gp$. Then \emph{(logarithmic) coordinates or parameters} of $X$ at $p$ are given by a finite set $x_1,\hdots, x_n\in \cO_{X,p}$ of elements that give a system of parameters (i.e. a regular sequence of maximal length) in $\cO_{S,p}$, together with an embedding $u\colon \overline{M}_p\hookrightarrow \cO_{X,p}$ lifting to the inclusion $M_p\subseteq \cO_{X,p}$ (note that $\overline{M}^\gp_p$ is a finitely generated free abelian group).

\'Etale locally around $p$ there is a strict classically smooth morphism $X\to \TV{\overline{M}_p}$ sending $p$ to the torus fixed point of $\TV{\overline{M}_p}$. By smoothness, \'etale locally around $p$ this map can be identified with the projection $\bA^n\times \TV{\overline{M}_p}\to \TV{\overline{M}_p}$. The coordinates $x_1,\hdots,x_n \in \cO_{X,p}$ correspond then to coordinates of $\bA^n$ in this local picture.

\begin{Prop}
Let $X$ be a smooth log scheme, $p \in X$ a closed point, and $x_1,\hdots, x_n\in \cO_{X,p}$ and $u\colon  \overline{M}_p\hookrightarrow \cO_{X,p}$ logarithmic coordinates at $p$. Then there is a canonical isomorphism
$$
\cT_{X,p}\cong \mathfrak{D}(N_p)\oplus \left( \bigoplus_{i} \cO_{X,p}\cdot \frac{\partial}{\partial x_i} \right)
$$
where 
\begin{itemize}
\item $\frac{\partial}{\partial x_i}$ is the unique section of $\cT_{X,p}$ that is zero on $u(\overline{M}_p)$ and such that $\frac{\partial}{\partial x_i}x_j=\delta_{ij}$, and
\item $N_p=\Hom_{\bZ}(\overline{M}_p^\gp,\cO_{X,p})$ and for $L\in N_p$, the element $\mathfrak{D}_L\in \cT_{X,p}$ is the unique derivation such that $\mathfrak{D}_L(u(m))=L(m)\cdot u(m)$ for every $m\in \overline{M}_p$ and $\mathfrak{D}_L(x_i)=0$ for all $i$.
\end{itemize}
\end{Prop}

The map $\mathfrak{D}\colon N_p \to \cT_{X,p}$ is denoted by $\bD$ in \cite{AbramovichTemkinWlodarczyk:arXiv:PrincipalizationOfIdealsOnToroidalOrbifolds}. More precisely, from the proof of Lemma 3.3.4 in \cite{AbramovichTemkinWlodarczyk:arXiv:PrincipalizationOfIdealsOnToroidalOrbifolds} it follows that if $m_1,\hdots, m_k$ are elements of $\overline{M}_p$ that form a basis of $\overline{M}_p^\gp$, then $\cT_{X,p}$ is freely generated by $\frac{\partial}{\partial x_1},\hdots, \frac{\partial}{\partial x_n}$ and $\partial \log(u(m_1)),\hdots, \partial \log(u(m_k))$, where $\partial \log(u(m_i)) = \text{\enquote{$m_i\frac{\partial}{\partial m_i}$}}$ is the element of $\mathfrak{D}(N_p)$ associated to the homomorphism in $N_p$ sending $m_i$ to $1$ and every other $m_j$ to $0$.

In order to keep the notation light, we will denote $\frac{\partial}{\partial x_i}$ by $\partial_i$ and $\partial \log(u(m_i))$ by $\partial_{m_i}$. Note that since $\cT_{X}$ is locally free, we can extend generators of this presentation to a neighbourhood of $p$.

\begin{Cor}\label{cor:local.description.differentials}
Let $X$ be a smooth log scheme, $p\in X$ a closed point, and choose logarithmic coordinates around $p$ as above.

Then \'etale locally around $p$, the sheaf $\cD_X$ can be described as the (non-commutative) algebra generated by $x_i$ for $0\leq i\leq n$, by $t^m$ with $m\in \overline{M}_p$, by symbols $\partial_i$ for $0\leq i \leq n$ and $\partial_{m_i}$, where $m_1,\hdots, m_k$ are elements of $\overline{M}_p$ that form a basis of $\overline{M}_p^\gp$, subject to the following relations:
\begin{align*}
[x_i,x_j]=0 & \;\;\;\;  \text{ for all } \;\; 0\leq i,j\leq n\\
[t^m,t^{m'}]=0 &  \;\;\;\;  \text{ for all }  \;\; m,m' \in \overline{M}_p \\
t^{m+m'}=t^m\cdot t^{m'} &  \;\;\;\;  \text{ for all }  \;\; m,m'\in \overline{M}_p\\
[\partial_i,\partial_j]=0 & \;\;\;\;  \text{ for all } \;\;  1\leq i,j\leq n\\
[\partial_i,\partial_{m_j}]=0 & \;\;\;\;  \text{ for all } \;\; 0\leq i\leq n  \; \text{ and } \; 0\leq j\leq k\\
[\partial_{m_i},\partial_{m_j}]=0 & \;\;\;\;  \text{ for all } \;\;  1\leq i,j\leq k\\
[\partial_i, x_j]=\delta_{ij} & \;\;\;\;  \text{ for all } \;\; 0\leq i,j\leq n \\
[\partial_{m_i},t^{m}]= a_i t^m & \;\;\;\;  \text{ for all } \;\; 0\leq i\leq k \; \text{ and } \; m=\, \scriptstyle\sum_i \displaystyle a_i m_i \in \overline{M}_p \\
[\partial_i, t^m]=0 & \;\;\;\;  \text{ for all } \;\; 0\leq i\leq n \; \text{ and } \; m\in \overline{M}_p\\
[\partial_{m_i}, x_j]=0 & \;\;\;\;  \text{ for all } \;\; 0\leq i\leq k  \; \text{ and } \; 0\leq j\leq n\\
\end{align*}
\end{Cor}

The variables $x_i$ and $t^m$ of the statement should of course be identified with the generators of the algebra $\kk[x_1,\hdots, x_n, \overline{M}_p]$, that describes the structure sheaf $\cO_X$ close to the point $p$.

\begin{Ex}\label{example:diff}
If $X=\bA^1=\Spec \kk[x]$ with the toric log structure, then $\cD_X$ is generated by $x$ and $\partial$, with $[\partial,x]=x$. In more standard notation, $\partial$ should be denoted by $x\frac{\partial}{\partial x}$.
\end{Ex}

If $X$ is idealized, we can define $\lD_X$ as follows: \'etale locally around every point we can embed $X$ in a toric variety $\TV{P}$, and restrict the sheaf $\lD_{\TV{P}}$ to $X$. One can check that these locally defined sheaves are canonically isomorphic on double intersections of the fixed \'etale cover of $X$, and by \'etale descent for quasi-coherent sheaves of algebras we obtain a global sheaf $\lD_X$. This construction is justified by Proposition \ref{prop:differential restriction}.

Locally around a closed point $p\in X$ we can find a classically smooth morphism  $X\to \ITV{P}{K}$ for $P=\overline{M}_p$ and some ideal $K\subseteq P$, and hence \'etale locally around $p$ we can identify $X$ with $\bA^n\times \ITV{P}{K}$ for some $n$. In this case, logarithmic coordinates at $p$ are given by the coordinates $x_1,\hdots, x_n$ of $\bA^n$ (corresponding to elements of $\cO_{X,p}$ giving a system of parameters of $\cO_{S,p}$, where $S$ is the locally closed log stratum through $p$), and by a map $u\colon \overline{M}_p\to \cO_{X,p}$ lifting to $M_p\to \cO_{X,p}$ (that in this case need not be injective).

As a consequence of Proposition \ref{prop:differential restriction}, we have:

\begin{Cor}\label{cor:local.description.differentials.idealized}
Let $X$ be a smooth idealized log scheme, $p\in X$ a closed point, and choose logarithmic coordinates around $p$ as above. Moreover let $K\subseteq \overline{M}_p$ be the ideal defining the idealized structure on $X$.

Then \'etale locally around $p$, the sheaf $\cD_X$ can be described as the (non-commutative) algebra generated by $x_i$ for $0\leq i\leq n$, by $t^m$ with $m\in \overline{M}_p$, by symbols $\partial_i$ for $0\leq i \leq n$ and $\partial_{m_i}$, where $m_1,\hdots, m_k$ are elements of $\overline{M}_p$ that form a basis of $\overline{M}_p^\gp$, subject to all the relations as in Corollary \ref{cor:local.description.differentials}, and also to $t^k=0$ for $k\in K$.
\end{Cor}

\begin{proof}
This is an immediate combination of Proposition \ref{prop:differential restriction} and Corollary \ref{cor:local.description.differentials}.
\end{proof}

\begin{Ex}
If $X$ is the standard log point, then $\cD_X$ is a $\kk$-algebra, freely generated by a single element $\partial$. Of course, $\partial$ is the restriction to the origin of the generator $x\frac{\partial}{\partial x}$ of Example \ref{example:diff}. In particular, in contrast with the situation for usual differentials, the element $x\frac{\partial}{\partial x}$ does not become zero when $x=0$.
\end{Ex}

In analogy with the classical case, we will endow $\lD_X$ with the \emph{order filtration}, i.e.~the filtration for which the generators $x_i$ and $t^m$ are in degree $0$ and $\partial_i$ and $\partial_{m_j}$ are in degree $1$.
This makes $\lD_X$ into a quasi-coherent differential $\O_X$-algebra in the sense of Section~\ref{sec:almost_commutative}.
The corresponding associated graded is isomorphic to $\pi _{*}\O_{T^*X}$, where $\pi \colon T^*X \to  X$ is the log cotangent bundle. In particular it follows that $\lD_X$ is left and right noetherian (cf.~\cite[Proposition~D.1.4]{HottaTakeuchiTanisaki:2008:DModulesPerverseSheavesRepresentationTheory}).

\subsubsection{Logarithmic D-modules}

Let $X$ be a separated smooth idealized log scheme of finite type over $\kk$.
A \emph{logarithmic D-module} on $X$ will be a sheaf $\sheaf F$ of left modules for the sheaf of algebras $\cD_X$. A homomorphism of log D-modules is a homomorphism of sheaves of left $\cD_X$-modules. We will denote the category of log D-modules on $X$ by $\catDMod{X}$.
In the rest of the paper we will suppress the word ``logarithmic'', and just use ``D-module'' or ``$\lD_X$-module''. Ordinary D-modules will not play any role, so there is no risk of confusion.

As in the classical case, we have the following equivalent characterization of D-modules:

\begin{Prop}\label{prop:D-mod}
Let $\sheaf F$ be a sheaf of $\cO_X$-modules. Then a structure of D-module on $\sheaf F$ corresponds to a $\kk$-linear morphism $\nabla\colon \cT_X\to \mathrm{End}_\kk(\sheaf F)$ satisfying:
\begin{enumerate}
\item $\nabla_{f \theta}(s)=f\nabla_ \theta(s)$ \hspace{2cm} for $f\in \cO_X, \theta\in \cT_X, s\in \sheaf F$,
\item $\nabla_ \theta(fs)= \theta(f)s+f\nabla_ \theta(s)$ \hspace{.67cm} for $f\in \cO_X,  \theta\in \cT_X, s\in \sheaf F$,
\item $\nabla_{[ \theta_1, \theta_2]}(s)=[\nabla_{ \theta_1},\nabla_{ \theta_2}](s)$  \hspace{.6cm}  for $ \theta_1, \theta_2 \in \cT_X, s\in \sheaf F$.
\end{enumerate}
\end{Prop}

There is also a category of right $\lD_X$-modules, usually identified with $\cat{Mod}(\lD_X^\mathrm{op})$, and one can switch between left and right $\lD_X$-modules by tensoring with the log canonical bundle.

In fact, $\omega_X$ has a natural structure of a right $\lD_X$-module via the \emph{Lie derivative}: for $\theta\in \cT_X$, define $(\mathrm{Lie}\, \theta)\omega$ for $\omega \in \omega_X$ by
\[
((\mathrm{Lie}\, \theta)\omega) (\theta_1,\hdots, \theta_n)=\theta(\omega(\theta_1,\hdots,\theta_n))-\sum_{i=1}^n \omega(\theta_1,\hdots, [\theta,\theta_i],\hdots, \theta_n)
\]
where $\theta_i\in \cT_X$, and $n=\logdim X$. The identities
\begin{itemize}
\item $(\mathrm{Lie}\, [\theta_1,\theta_2])\omega=(\mathrm{Lie}\,\theta_1)((\mathrm{Lie}\,\theta_2)\omega)-(\mathrm{Lie}\,\theta_2)((\mathrm{Lie}\,\theta_1)\omega)$,
\item $(\mathrm{Lie}\,\theta)(f\omega)=f((\mathrm{Lie}\,\theta)\omega)+\theta(f)\omega$,
\item $(\mathrm{Lie}\,f\theta)\omega=(\mathrm{Lie}\,\theta)(f\omega)$
\end{itemize}
show (using the analogue for right $\lD_X$-modules of Proposition \ref{prop:D-mod}) that $\omega\theta=-(\mathrm{Lie}\, \theta)\omega$ defines a structure of right $\lD_X$-module on $\omega_X$.

For a left $\lD_X$-module $\sheaf F$, the tensor product $\omega_X\otimes_{\cO_X} \sheaf F$ can be endowed with a structure of right $\lD_X$-module: for sections $\omega\in \omega_X, s\in \sheaf F$ and $\theta\in \cT_X$, set $(\omega\otimes s)\theta=\omega\theta\otimes s-\omega\otimes \theta s$.
\begin{Lem}
    The functor $\sheaf F\mapsto \omega_X\otimes_{\cO_X} \sheaf F$ from $\cat{Mod}(\lD_X)$ to $\cat{Mod}(\lD_X^\mathrm{op})$ is an equivalence of categories with quasi-inverse given by $\sheaf G \mapsto \omega_X^\vee\otimes_{\cO_X} \sheaf G =\sheafHom_{\cO_X}(\omega_X,\,\sheaf G)$.
\end{Lem}
These functors are sometimes called the \emph{side-changing operations}.

\begin{proof}
    The action of $\theta \in \cT_X$ on a section $\phi \in \sheafHom_{\cO_X}(\omega_X,\, \sheaf G)$ is given by $(\theta\phi)(s) = -\phi(s)\theta + \phi(s\theta)$.
    Exactly as in the classical case one immediately checks that the two functors are quasi-inverse to each other.
\end{proof}

We call a $\lD_X$-module \emph{quasi-coherent} if it is so as an $\O_X$-module.
A $\lD_X$-module $\sheaf F$ will be called \emph{coherent} if locally on $X$, $\sheaf F$ can be written as a cokernel
\[
\cD_X^n\to \cD_X^m\to \sheaf F\to 0
\]
of a morphism of free $\cD_X$-modules of finite rank. Since $\cD_X$ is noetherian, this notion of coherence is equivalent to the stronger one, requiring that $\sheaf F$ be locally finitely generated, and that every local finitely generated subsheaf of $\sheaf F$ be finitely presented.
We write $\catDModcoh{X}$ for the category of coherent D-modules and $\catDbcohDMod{X}$ for the full subcategory of $\catDDMod{X}$ consisting of complexes with finitely many nonzero cohomology sheaves, all of which are coherent.
As in the classical case, one proves that this is the same as the bounded derived category of $\catDModcoh{X}$, cf.~\cite[Proposition~\RomanNum{VI}.2.11]{Borel:1987:AlgebraicDModules}.
This follows from the following lemma, which again can be proven as in the classical case, cf.~\cite[Lemma~\RomanNum{VI}.2.3]{Borel:1987:AlgebraicDModules}.
\begin{Lem}\label{lem:coherent_subs}
  Let $\sheaf F$ be a quasi-coherent $\lD_X$-module.
  Then the following hold:
  \begin{enumerate}
    \item If $\sheaf F$ is coherent, then it is generated as a $\lD_X$-module by an $\O_X$-coherent $\O_X$-submodule.
    \item Let $U \subseteq X$ be open and such that $\res{\sheaf F}U$ is $\lD_U$-coherent. Then there exists a $\lD_X$-coherent submodule $\widetilde{\sheaf F}$ such that $\res{\widetilde{\sheaf F}}{U} = \res{\sheaf F}{U}$.
    \item $\sheaf F$ is the filtered colimit of its $\lD_X$-coherent submodules.
      \qed
  \end{enumerate}
\end{Lem}

If $\sheaf F \in  \catDModcoh{X}$ is a coherent D-module we can choose a good filtration on it.
As in the classical case, Lemma~\ref{lem:coherent_subs} implies that one can choose a global filtration, cf.~\cite[Theorem~2.1.3]{HottaTakeuchiTanisaki:2008:DModulesPerverseSheavesRepresentationTheory}.
The associated graded is then a coherent $\gr \lD_X$-module and its support is called the \emph{characteristic variety} $\Ch(\sheaf F) \subseteq  T^*X$ of $\sheaf F$.
We will only ever consider $\Ch(\sheaf F)$ as a topological space and hence always endow it with the induced reduced structure.
Let us remark that $\dim T^*X = \dim X + \logdim X$.

\begin{Rem}
    If $\Ch(\sheaf F)$ is contained in the zero section $T^*_XX$ of the cotangent bundle, then as in the proof of \cite[Proposition~2.2.5]{HottaTakeuchiTanisaki:2008:DModulesPerverseSheavesRepresentationTheory} any good filtration of $\sheaf F$ must stabilize after finitely many steps and hence $\sheaf F$ is $\O_X$-coherent.
    Thus the subcategory of $\catDModcoh{X}$ of modules with characteristic variety contained in $T^*_XX$ is the same as Ogus's category $MIC_{coh}(X/\mathbb{C})$ of coherent sheaves equipped with a log connection \cite{Ogus:2003:LogarithmicRiemannHilbertCorrespondence}.
\end{Rem}

\subsubsection{Duality}\label{sec:duality}
We fix a rigid dualizing complex $\sheaf R$ for $\lD_X$ as in Section~\ref{sec:almost_commutative}.
If the underlying classical scheme $\underline X$ is smooth, then it has been shown in \cite{Chemla:2004:RigidDualizingComplexForQuantumEnvelopingAlgebras} that $\sheaf R \cong \lD_X \otimes_{\O_X} \Hom_{\O_X}(\canon[X], \canon[\underline X])[\dim X + \logdim X]$.
We will show in future work how to extend this formula to general log smooth, but not necessarily smooth, varieties (even though in that case $\canon[\underline X]$ might not be a line bundle).
In order to match the classical (i.e.~trivial log structure) setting, we introduce a shift in the duality functor and define the dual of any $\sheaf F \in  \catDDMod{X}$ as
\begin{align*}
  \DVerdier_X \sheaf F
  & = \sheafHom_{\lD_X}(\sheaf F, \sheaf R)[-\logdim X] \otimes_{\O_X} \canon[X]^{\vee} \\
  & = \sheafHom_{\lD_X}(\sheaf F, \sheaf R \otimes_{\O_X} \canon[X]^{\vee}[-\logdim X]),
\end{align*}
where we tensor with $\omega^{\vee}_X$ to obtain a functor of left $\lD_X$-modules.
By definition $\DVerdier_X$ restricts to an involutive anti-autoequivalence of $\catDbcohDMod{X}$.

\begin{Rem}
    It might be confusing that we have to introduce a shift by $-\logdim X$ (compared to $\dim X$ in the classical setting).
    This is because if $X$ is a smooth complex variety, then the rigid dualizing complex is already concentrated in cohomological degree $-2\dim X$.
    For good behaviour in the idealized case the shift by $-\logdim X$ proves to be more natural than the alternative of a shift by $-\dim X$.
    For example on the standard log point, the dualizing complex $\sheaf R$ is $\lD_X[1]$, so that $\DVerdier \O_X = \O_X[-1]$.
    If $p$ is the canonical map from the log point to the usual point then one easily computes that the pushforward along $p$ (see Section~\ref{sec:pf}) of $\O_X$ is the cohomology of $S^1$, shifted down by $1$.
    Thus our convention ensures (at least in this case) that pushworward along proper maps commutes with duality.
\end{Rem}

\begin{Lem}\label{lem:dual_of_hom}
    For any $\sheaf F,\sheaf G \in  \catDbcohDMod{X}$ there exists an isomorphism of $\kk_X$-modules
    \[
        \sheafHom_{\O_X}(\sheaf F, \sheaf G) \cong \DVerdier_X \bigl(\sheaf F \otimes_{\O_X} \DVerdier_X\sheaf G\bigr).
    \]
\end{Lem}
  
We note that here $\sheafHom_{\O_X}(\sheaf F, \sheaf G)$ is a left $\lD_X$-module by the usual construction, cf.~\cite[Proposition~1.2.9]{HottaTakeuchiTanisaki:2008:DModulesPerverseSheavesRepresentationTheory}.

\begin{proof}
    By Tensor-Hom adjunction we have
    \begin{align*}
        \DVerdier_X(\sheaf F \otimes_{\O_X} \DVerdier_X \sheaf G) &\cong
        \sheafHom_{\lD_X}(\sheaf F \otimes_{\O_X} \DVerdier_X \sheaf G,\, \sheaf R \otimes_{\O_X} \canon[X]^{\vee}[-\logdim X]) \\ &\cong
        \sheafHom_{\O_X}(\sheaf F,\, \DVerdier_X\DVerdier_X\sheaf G) \\ & \cong
        \sheafHom_{\O_X}(\sheaf F,\, \sheaf G).
        \qedhere
    \end{align*}
\end{proof}

\subsubsection{Pullback and pushforward}\label{sec:pf}

In the remainder of this subsection all functors will be underived unless noted otherwise.

Let $f\colon X \to Y$ be a morphism of smooth idealized log schemes.
The definition of the pullback of a $\lD_Y$-module along $f$ exactly mirrors the definition of the pullback in the classical setting.
Let us give an explicit description.

For a $\lD_X$-module $\sheaf F \in \catDMod{Y}$ let
\[
    f^\epb \sheaf F = \O_X \otimes_{f^{-1}\O_Y} f^{-1}\sheaf F,
\]
be the $\O$-module pullback.
In order to define a $\lD_X$-module structure on $f^\epb \sheaf F$, it suffices to specify the action $\nabla_\theta$ of each $\theta \in \cT_X$ (Proposition \ref{prop:D-mod}).
For this we use the map
\[
    \cT_X \to f^*\cT_Y,\quad \theta \mapsto \tilde\theta
\]
dual to the canonical map $f^*\Omega_Y^1 \to \Omega_X^1$.
Then for $\theta \in \cT_X$ we define
\[
    \theta(\psi \otimes s) = \theta(\psi) \otimes s + \psi\tilde\theta(s), \qquad \psi \in \O_X,\, s \in \sheaf F.
\]
Here if $\tilde \theta = \sum \phi_j \otimes \theta_j$ (where $\phi_j \in \O_X$, $\theta_j \in \cT_Y$) we set $\psi\tilde\theta(s) = \sum \psi\phi_j \otimes \theta_j(s)$.
In particular this endows the sheaf
\[
    \lD_{X \to Y} = f^\epb\lD_Y = \O_X \otimes_{f^{-1}\O_Y} f^{-1}\lD_Y
\]
with the structure of a $(\lD_X,\, f^{-1}\lD_Y)$-bimodule and we have an isomorphism of $\lD_X$-modules
\[
    f^\epb\sheaf F \cong \lD_{X\to Y} \otimes_{f^{-1}\lD_Y} f^{-1}\sheaf F.
\]

\begin{Def}
    For a morphism $f\colon X \to Y$ of smooth idealized log schemes we define the \emph{naive pullback} $f^*\colon \catDDMod[-]{Y} \to \catDDMod[-]{X}$ by
    \[
        \sheaf M \mapsto \lD_{X\to Y} \Lotimes_{f^{-1}\lD_Y} f^{-1}(\sheaf M).
    \]
\end{Def}

\begin{Rem}
    Unlike in the classical setting, the !-pullback $f^!$ is not simply a shift of the naive pullback.
    To see why this is the case let us consider a smooth log variety $X$ with structure map $p\colon X \to \mathrm{pt}$.
    Then it is easy to see that the pushforward $p_\spf\colon \catDMod{X} \to \cat{Vect}$ (see Definition~\ref{def:spf}) is represented by $\O_X$.
    Thus the six-functor formalism would imply that (up to shift) $f^!k = \DVerdier \O_X$.
    On the other hand, at least if the underlying variety $\underline X$ is smooth, then by the discussion in Section~\ref{sec:duality} $\DVerdier_X\O_X = \sheafHom_{\O_X}(\omega_X,\, \omega_{\underline X})$ which in general differs from $f^*k = \O_X$. 

    For the purpose of the present work the naive pullback will be sufficient.
\end{Rem}

Similarly, the pushforward of a $\lD_X$-module along $f$ can again be defined as in the classical case.
Thus, given a \emph{right} D-module $\sheaf F$, we get a right $f^{-1}{\lD_Y}$-module $\sheaf F \otimes_{\lD_X} \lD_{X \to Y}$ and hence a right $\lD_Y$-module $f_*(\sheaf F \otimes_{\lD_X} \lD_{X \to Y})$ (where $f_*$ is the sheaf theoretic direct image functor).
In order to get a functor of \emph{left} D-modules we apply the side-changing operations.
Thus we arrive at
\begin{multline*}
    \canon[Y]^\vee \otimes_{\O_Y} f_*((\canon[X] \otimes_{\O_X} \sheaf F) \otimes_{\lD_X} \lD_{X \to Y})  
     \cong \canon[Y]^{\vee} \otimes_{\O_Y} f_*((\canon[X] \otimes_{\O_X} \lD_{X\to Y}) \otimes_{\lD_X} \sheaf F)  \\
     \cong f_*((\canon[X] \otimes_{\O_X} \lD_{X\to Y} \otimes_{f^{-1}\O_Y} f^{-1}\canon[Y]^{\vee}) \otimes_{\lD_X} \sheaf F),
\end{multline*}
where the first isomorphism follows by the analogue of \cite[Lemma~1.2.11]{HottaTakeuchiTanisaki:2008:DModulesPerverseSheavesRepresentationTheory}.
As in the classical case we define the $(f^{-1}\lD_Y,\, \lD_X)$-bimodule 
\[
    \lD_{Y \from X} =
    \canon[X] \otimes_{\O_X} \lD_{X\to Y} \otimes_{f^{-1}\O_Y} f^{-1}\canon[Y]^\vee.
\]

\begin{Def}\label{def:spf}
    For a morphism $f\colon X \to Y$ of smooth idealized log schemes we have a, potentially only partially defined, direct image functor $f_\spf\colon \catDbDMod{X} \to \catDDMod{Y}$ given by
    \[
        \sheaf F \mapsto f_\spf\sheaf F = \RR f_*(\lD_{Y \from X} \Lotimes_{\lD_X} \sheaf F).
    \]
\end{Def}

\begin{Rem}
    Since the underlying classical scheme $\underline{X}$ might have singularities, the ring $\lD_X$ does not necessarily need to be of finite global dimension.
    Thus the result of the tensor product in the above definition might not be bounded below, so that we cannot apply the right derived functor of $f_*$ to it.
    In this paper we will only consider situations where the pushforward is defined on the whole category.
    We will return to this issue in future work.
\end{Rem}

\subsection{Restrictions to log strata}\label{sec:Dmods-and-strata}

The purpose of this section is to prove a version of \cite[Proposition~5.2]{Kashiwara:2004:tStructureOnHolonomicDModuleCoherentOModules}.
In effect we can view this as obtaining a refinement of Corollary~\ref{cor:nc_dual_vanishing_and_support_derived} for local cohomology along log strata. 

Let $i\colon X^k \to  X$ be the inclusion of a log stratum, endowed with the induced idealized log structure.
Let $\sheaf I$ be the sheaf ideals defining $X^k$.
By \cite[Lemma~3.3.4(4)]{AbramovichTemkinWlodarczyk:arXiv:PrincipalizationOfIdealsOnToroidalOrbifolds}), the sheaf $\sheaf I$, and hence also $\O_{X^k} = \O_X/\sheaf I$, has a canonical $\lD_X$-module structure.
From the presentation of $\lD_X$ discussed in (\ref{sec:higher.differential}) we have canonical isomorphisms
\[
    \lD_{X^k \to  X} \cong \O_{X^k} \otimes_{\O_X} \lD_X \cong \lD_{X^k}.
\]
Similarly, recall that we have $\canon[X^k] \cong \O_{X^k} \otimes_{\O_X} \canon[X]$.
Thus the pushforward $i_{\spf}$ is well-defined and coincides with the $\O$-module pushforward.
In particular we have for any $\sheaf F \in  \catDDMod{X}$,
\[
    i_{\spf}i^* \sheaf F \cong \O_{X^k} \otimes_{\O_X} \sheaf F,
\]
where the latter has a natural structure of left $\lD_X$-module.

\begin{Rem}
    It follows from this that Kashiwara's equivalence for $\lD_X$-modules supported on a closed subvariety cannot hold for general log varieties.
    As an example, consider the log line $X = \as 1$ and let $i$ be the inclusion of the origin endowed with the induced idealized log structure.
    Then for example $i_\spf i^* (\O_X/x^n\O_X)$ is not quasi-isomorphic to $\O_X/x^n\O_X$.
    Moreover for $n > 0$ there cannot be any sheaf $\sheaf F$ whose pushforward is isomorphic to $\O_X/x^n\O_X$ even as $\O_X$-modules.
    The situation is even worse if $i$ is the inclusion of the origin with trivial log structure, as then $i^*$ will forget all information about the action of $x\frac{\partial}{\partial x}$ on skyscraper sheaves at the origin.
\end{Rem}

By Proposition \ref{prop:differential restriction} we have an identification of log cotangent spaces $T^*X^k = \res{T^*X}{X^k}=T^*X \times _{X} X^k\subseteq T^*X$.
For any subspace $Y$ of $T^*X$ we set $\res{Y}{X^k} = Y \cap  T^*X^k$.
Consider now the short exact sequence of right modules
\[
    0 \to \sheaf I\lD_X \to \lD_X \to \rquot{\lD_X}{\sheaf I\lD_X} \to 0.
\]
By the above, we know that $\rquot{\lD_X}{\sheaf I\lD_X} \cong \lD_{X^k}$, compatibly with the filtrations.
Thus the induced short exact sequence
\[
    0 \to \gr(\sheaf I\lD_X) \to \gr \lD_X \to \gr\lD_{X^k} \to 0.
\]
shows that $\gr(\sheaf I\lD_X)$ is the sheaf of ideals of $\O_{T^*X}$ defining $T^*X^k$.

\begin{Lem}\label{lem:ch_of_restriction_to_stratum}
    Let $\sheaf F$ be a coherent $\lD_X$-module and let $i\colon X^k \hookrightarrow X$ be the inclusion of a log stratum.
    Then, under the identification $T^*X^k = \res{T^*X}{X^k}$ we have
    \begin{enumerate}
        \item\label{lem:ch_of_restriction_to_stratum:i}
            $\Ch(H^0(i^*\sheaf F)) = \res{\Ch(\sheaf F)}{X^k}$;
        \item\label{lem:ch_of_restriction_to_stratum:ii}
            $\Ch(H^\ell(i^*\sheaf F)) \subseteq \res{\Ch(\sheaf F)}{X^k}$ for all $\ell \in  \mathbb{Z}$.
    \end{enumerate}
\end{Lem}

\begin{proof}
    Let $\sheaf I$ be the sheaf of ideals defining $X^k$.
    As noted above, $\sheaf I$ is a coherent $\lD_X$-module and hence the same is true for $\sheaf I\sheaf F$.
    Consider now the short exact sequence of $\lD_X$-modules
    \[
        0 \to  \sheaf I\sheaf F \to  \sheaf F \to  \sheaf F/\sheaf I\sheaf F \to  0.
    \]
    The module $\sheaf I\sheaf F$ is the same as $(\sheaf I\lD_X)\sheaf F$, where $(\sheaf I\lD_X)$ is a sheaf of (right) ideals of $\lD_X$.
    We have $\gr(\sheaf I\sheaf F) \cong \gr(\sheaf I\lD_X)\gr\sheaf F$.
    Thus we get a short exact sequence
    \[
        0 \to  \gr(\sheaf I\lD_X)\gr\sheaf F \to  \gr \sheaf F \to  \gr(\sheaf F/\sheaf I\sheaf F) \to  0,
    \]
    and hence an isomorphism
    \[
        \gr(\sheaf F/\sheaf I\sheaf F) \cong \rquot{\gr\sheaf F}{\gr(\sheaf I\lD_X)\gr\sheaf F} \cong \gr \sheaf F \otimes_{\O_{T^*X}} \O_{T^*X^k},
    \]
    showing the first statement.

    For the second statement it now suffices to show that the characteristic varieties of $H^i(\sheaf F \otimes_{\O_X} \O_{X^k})$ are contained in $\Ch(\sheaf F)$.
    Since the statement is local, we can assume that $X$ is affine and hence has the resolution property (i.e. every coherent sheaf admits a surjection from a vector bundle).
    Thus we can compute the tensor product via a locally free resolution of $\O_{X^k}$, which implies that $H^i(\sheaf F \otimes_{\O_X} \O_{X^k})$ is locally a subquotient of a direct sum of copies of $\sheaf F$.
    The statement follows. 
\end{proof}

It is well known that for a coherent $\O_X$-module $\sheaf F$ with $\supp(\sheaf F) = Z \cup  Z'$ with $Z,Z'$ closed there exists a short exact sequence of coherent sheaves $0\to \sheaf G \to  \sheaf F \to  \sheaf G' \to  0$ with $\supp(\sheaf G) \subseteq  Z$ and $\supp \sheaf G' \subseteq  Z'$ \stackcite{01YC}.
The following statement is an adaptation of this to $\lD_X$-modules, but only when one of the supports is a log stratum.

\begin{Lem}\label{lem:ses_by_support}
    Let $\sheaf F \in  \catDModcoh{X}$ and assume that there exists a closed subscheme $Z$ of $T^*X$ such that $\Ch(\sheaf F)$ is contained in $\res{T^*X}{X^k} \cup  Z$ for some $k$.
    Then there exist a short exact sequence of coherent $\lD_X$-modules
    \[
        0 \to  \sheaf G \to  \sheaf F \to  \sheaf G' \to  0
    \]
    with $\Ch(\sheaf G') \subseteq  \res{T^*X}{X^k}$ and $\Ch(\sheaf G) \subseteq  Z$.
\end{Lem}

\begin{proof}
    All functors in this proof will be between the abelian categories, i.e.~underived.
    Let $\sheaf I \subseteq \O_X$ be the sheaf of ideals defining the reduced subscheme structure on $X^k$.
    Set $\sheaf G'_n = (\O_X/\sheaf I^n) \otimes_{\O_X} \sheaf F = (\lD_X/\sheaf I^n\lD_X) \otimes_{\lD_X} \sheaf F$ and $\sheaf G_n = \sheaf I^n \sheaf F = \ker (\sheaf F \to  \sheaf G'_n)$.
    That is, for each $n$ we have a short exact sequence
    \[
        0 \to  \sheaf G_n \to  \sheaf F \to  \sheaf G_n' \to  0.
    \]
    Clearly $\supp(\sheaf G_n') \subseteq  X^k$ and hence $\Ch(\sheaf G_n') \subseteq  \res{T^*X}{X^k}$.
    We have to show that $\Ch(\sheaf G_n) \subseteq  Z$ for $n \gg 0$ or equivalently that $\res{\gr \sheaf G_n}{T^*X \setminus Z} = 0$.

    By the proof of Lemma~\ref{lem:ch_of_restriction_to_stratum}, we have
    \[
        \gr \sheaf G'_n \cong (\O_{T^*X}/\sheaf I^n\O_{T^*X}) \otimes_{\O_{T^*X}} \gr \sheaf F.
    \]
    Thus the short exact sequence
    \[
        0 \to  \gr\sheaf G_n \to  \gr\sheaf F \to  \gr\sheaf G_n' \to  0
    \]
    shows that $\gr\sheaf G_n \cong (\sheaf I^n\O_{T^*X})\gr\sheaf F$.
    Thus also $\res{\gr\sheaf G_n}{T^*X \setminus Z} \cong \res{(\sheaf I^n\O_{T^*X})\gr\sheaf F}{T^*X \setminus Z}$.
    As $\res{\gr\sheaf F}{T^*X \setminus Z}$ is supported on $\res{T^*X}{X^k}$, $\res{(\sheaf I^n\O_{T^*X})\gr\sheaf F}{T^*X \setminus Z}$ has to vanish for $n \gg 0$ \stackcite{01Y9}.
    In other words, $\supp\gr\sheaf G_n \subseteq  Z$ for some $n$, as required.
\end{proof}

The following statements are adapted from similar statements in \cite[Section~5]{Kashiwara:2004:tStructureOnHolonomicDModuleCoherentOModules}, specifically from Proposition~5.2, Lemma~5.3 and Proposition~5.4.
We want to emphasize that they are only true for restrictions to the strata $X^k$.
The analogous statements for arbitrary closed subschemes fail.

\begin{Lem}\label{lem:lemma_for_Xk_local_cohomology_and_support}
  Let $\sheaf F \in  \catDcohDMod[-]{X}$.
  Then under the identification $T^*X^k \hookrightarrow T^*X$ we have an equality of characteristic varieties
  \[
    \res{\Ch(\taustand_{\ge\ell} \sheaf F)}{X^k} =
    \Ch\bigl(\taustand_{\ge\ell}(\O_X/\sheaf I \otimes_{\O_X} \sheaf F)\bigr),
  \]
 for every $\ell \in \mathbb{Z}$, where $\sheaf I$ is the sheaf of ideals defining the reduced subscheme structure on $X^k$.

\end{Lem}

\begin{proof}
  The inclusion $\Ch\bigl(\taustand_{\ge\ell}(\O_X/\sheaf I \otimes_{\O_X} \sheaf F)\bigr) \subseteq \res{\Ch(\taustand_{\ge\ell} \sheaf F)}{X^k}$ follows immediately by induction from Lemma~\ref{lem:ch_of_restriction_to_stratum}\ref{lem:ch_of_restriction_to_stratum:ii}.
  
  For the other inclusion fix any point $x \in  \res{\Ch(\taustand_{\ge \ell} \sheaf F)}{X^k}$ and let $j$ be the largest integer such that $x \in  \Ch(H^j(\sheaf F))$.
  Let $i\colon X^k\hookrightarrow X$ be the inclusion.
  By Lemma~\ref{lem:ch_of_restriction_to_stratum} we see that $x \in \Ch(H^0(i^* H^j(\sheaf F)))$.
  Hence from the spectral sequence
  \[
    H^{-p}(i^*H^q(\sheaf F)) \Rightarrow H^{p+q}(i^*\sheaf F)
  \]
  and the maximality of $j$ it follows that
  \[
    x \in \Ch(H^j(i^* \sheaf F)) =
    \Ch(H^j(\O_X/\sheaf I \otimes_{\O_X} \sheaf F)) \subseteq 
    \Ch(\taustand_{\ge \ell}(\O_X/\sheaf I \otimes_{\O_X} \sheaf F))
  \]
  as required.
\end{proof}

We are now ready to state and prove the main proposition of this section.

\begin{Prop}\label{prop:Xk_local_cohomology_and_support}
    Let $\sheaf F \in  \catDbcohDMod{X}$ and let $k$ and $n$ be integers.
    Then $\Gamma _{X^k}(\sheaf F) \in  \catDQCoh[\ge n]{\O_X}$ if and only if
    \[
    \dim\res{\Ch\bigl(H^\ell(\DVerdier_X\sheaf F)\bigr)}{X^k} \le \logdim X - n - \ell
    \]
    for all $\ell\in \mathbb{Z}$.
\end{Prop}

\begin{proof}
  By \cite[Proposition~\RomanNum{VII}.1.2]{SGA2}, $\Gamma _{X^k}(\sheaf F) \in  \catDQCoh[\ge n]{\O_X}$ if and only if
  \begin{equation*}%\label{eq:proof:prop:Xk_local_cohomology_and_support:1}
    \sheafHom_{\O_X}(\O_X/\sheaf I,\, \sheaf F) \in  \cat{D}^{\ge n}(\kk_X),
  \end{equation*}
  where $\sheaf I$ is the sheaf of ideals defining $X^k$ with the reduced subscheme structure.
  We note that by Lemma~\ref{lem:dual_of_hom} we have an isomorphism of $\kk_X$-modules
  \[
    \sheafHom_{\O_X}(\O_X/\sheaf I,\, \sheaf F) \cong
    \DVerdier_X\bigl(\O_X/\sheaf I \otimes_{\O_X} \DVerdier_X \sheaf F\bigr).
  \]
  Hence it suffices to show that the latter is contained in $\cat{D}^{\ge n}(\kk_X)$.
  By Corollary~\ref{cor:nc_dual_vanishing_and_support_unbounded} this is equivalent to
  \[
    \dim \Ch\biggl(H^\ell\bigl(\O_X/\sheaf I \otimes_{\O_X} \DVerdier_X \sheaf F\bigr)\biggr)
    \le
    \logdim X - n - \ell
  \]
  for all $\ell \in  \mathbb{Z}$.
  Thus the statement follows from Lemma~\ref{lem:lemma_for_Xk_local_cohomology_and_support}.
\end{proof}

\subsection{Holonomic log D-modules}\label{sec:holonomics}

In the classical case the Bernstein inequality says that $\dim \Ch(\sheaf F) \geq  \dim X$.
However this is no longer true in the logarithmic setting.
For example, the characteristic variety of the skyscraper $\mathbb{C}$ at the origin of the log line $\mathbb{A}^1$ (say with $x\frac{\partial }{\partial x}$ acting by $0$) is just the origin, and hence $0$-dimensional.
As the following proposition shows, the reason is that we should account for the logarithmic structure and use the log dimension to measure characteristic varieties.

Recall that we regard $T^*X$ as a log scheme, equipped with the pullback of the log structure of $X$ via the projection $\pi\colon T^*X\to X$.

\begin{Thm}[Logarithmic Bernstein inequality]\label{thm:log_bernstein}
    Let $\sheaf F$ be a coherent $\lD_X$-module.
    Then if $Z$ is any irreducible component of $\Ch(\sheaf F)$,
    \[
        \logdim Z \geq  \logdim X.
    \]
\end{Thm}

Note that if $X$ is not idealized, then $\logdim X=\dim X$, and the above inequality reads $ \logdim Z \geq  \dim X$.

\begin{proof}
    Let $k$ be the largest integer such that $Z \subseteq  T^*X^k$.
    By Lemma~\ref{lem:ch_of_restriction_to_stratum}\ref{lem:ch_of_restriction_to_stratum:i} we can replace $X$ by $X^k$ (with the induced idealized log structure) and $\sheaf F$ by its restriction to assume without loss of generality that $Z$ intersects the open log stratum $U$ non-trivially.
    
    It suffices to show that $\logdim \res{Z}{U} \ge \logdim U = \logdim X$.
    Thus we replace $X$ by $U$ and $\sheaf F$ by $\res{\sheaf F}{U}$ and assume that $X$ consists of a single stratum.
    In this case $X$ is classically smooth.
    Locally from the description of the ring of differential operators in Corollary~\ref{cor:local.description.differentials.idealized} we have an isomorphism
    \[
        \lD_{X} \cong \lD_{\underline{X}}[\partial_{m_1},\dotsc,\partial_{m_d}],
    \]
    where $\lD_{\underline{X}}$ is just the usual algebra of differential operators on the classical scheme $\underline{X}$ and $d = \logdim X - \dim X$.
    We note that the generators $\partial_{m_i}$ all commute with everything.
    
    From \cite[Theorem~3.4 on page~43]{Bjoerk:1979:RingsOfDifferentialOperators} it follows that the (left) global dimension of $\lD_{X}$ is
    \[
        \gldim \lD_{X} = \gldim \lD_{\underline{X}} + d = \dim X + d = \logdim X
    \]
    and
    \[
        \gldim \gr \lD_{X} = \gldim \gr\lD_{\underline{X}} + d = 2\dim X + d = \dim X + \logdim X.
    \]
    Thus from \cite[Corollary~7.2 on page~73]{Bjoerk:1979:RingsOfDifferentialOperators} we obtain that
    \[
        \dim Z \ge \gldim \gr \lD_{X} - \gldim \lD_{X} = \dim X
    \]
    and hence finally
    \[
        \logdim Z = \dim Z + d \ge \logdim X. \qedhere
    \]
\end{proof}

\begin{Def}
    A coherent $\lD_X$-module $\sheaf F$ is called \emph{holonomic} if $\logdim \Ch(\sheaf F) = \logdim X$.
    An element of $\catDbcohDMod{X}$ is called \emph{holonomic} if all its cohomology modules are holonomic.
\end{Def}

\begin{Rem}
    If the underlying classical scheme $\underline{X}$ is smooth then we have a forgetful functor from $\lD_{\underline X}$-modules to $\lD_X$-modules via the pullback along the canonical map $X \to \underline X$.
    If $X$ is a smooth log scheme, then this comes from the inclusion $\lD_{X} \subseteq \lD_{\underline{X}}$.
    This forgetful functor does however not preserve holonomicity, as the example of the skyscraper $\kk[\frac{\partial}{\partial x}]$ at the origin of $\as 1$ shows.
    On the other hand it is not too hard to see that the pushforward along the map $X \to \underline X$ does preserve holonomicity.
\end{Rem}

\begin{Prop}\label{prop:logdim_of_dual}
    For any $\sheaf F \in \catDbcohDMod{X}$ one has $\logdim \sheaf F = \logdim \DVerdier_X \sheaf F$.
    In particular $\sheaf F$ is holonomic if and only if $\DVerdier_X \sheaf F$ is holonomic.
\end{Prop}

\begin{proof}
    By duality it is sufficient to show that $\logdim \DVerdier_X \sheaf F \le \logdim \sheaf F$.
    As usual, we can assume that $\sheaf F \in \catDModcoh{X}$.
    It suffices to show that for any $k \ge 0$ we have $ \dim \res{\Ch(\DVerdier_X\sheaf F)}{X^k} \le \dim \res{\Ch(\sheaf F)}{X^k}$.
    By Lemma~\ref{lem:ch_of_restriction_to_stratum} the left hand side is equal to $\dim \Ch(\lD_{X^k} \otimes_{\lD_X} \DVerdier_X\sheaf F)$.
    There is a canonical isomorphism
    \[
        \lD_{X^k} \otimes_{\lD_X} \sheafHom_{\lD_X}(\sheaf F,\, \sheaf R) \cong
        \sheafHom_{\lD_X}(\sheafHom_{\lD_X}(\lD_{X^k},\, \sheaf F),\, \sheaf R),
    \]
    where $\sheaf R$ is the dualizing complex.
    As duality preserves the dimension of the characteristic variety (Lemma~\ref{lem:duality_keeps_dim_ch}) it thus is enough to show that $\Ch(\sheafHom_{\lD_X}(\lD_{X^k},\sheaf F))$ is contained in $\res{\Ch(\sheaf F)}{X^k}$.
    This follows from the fact that $\gr\sheafExt^i(\lD_{X^k},\, \sheaf F)$ is isomorphic to a subquotient of $\sheafExt^i_{\O_{T^*X}}(\O_{T^*X^k},\, \gr \sheaf F)$ \cite[Lemma~D.2.4]{HottaTakeuchiTanisaki:2008:DModulesPerverseSheavesRepresentationTheory}.
\end{proof}

\section{The log perverse t-structure}\label{sec:t-structure}

A fundamental property of holonomic D-modules in the classical setting is that their duals are concentrated in cohomological degree $0$.
Indeed, this property can be used to characterize the holonomic D-modules among all coherent ones.
As the following example shows, this is no longer true in the logarithmic setting.

\begin{Ex}\label{ex:duality_fails}
    Let $X = \as 1$ be the log line (with log structure given by the origin).
    The module $\O_X$ has a free resolution
    \[
        \lD_X \xrightarrow{\cdot x\frac{\partial}{\partial x}} \lD_X
    \]
    and hence dual (as right $\lD_X$-module and up to a twist by a line bundle)
    \[
        \sheafHom_{\lD_X}(\O_X, \lD_X)[1] = \bigl(\lD_X \xrightarrow{x\frac{\partial}{\partial x} \cdot} \lD_X\bigr)[1] = \O_X
    \]
    (where the complex now in in degrees $0$ and $1$, as opposed to $-1$ and $0$).
    On the other hand the skyscraper $\mathbb{C}_0$ at the origin with $x\frac{\partial }{\partial x}$ acting by zero has a free resolution
    \[
        \lD_X \xrightarrow{\cdot(x,\, 1 - x\frac{\partial}{\partial x})} \lD_X^2 \xrightarrow{\cdot\begin{psmallmatrix} x\frac{\partial}{\partial x}\\x\end{psmallmatrix}} \lD_X.
    \]
    Thus the dual is 
    \[
        \sheafHom_{\lD_X}(\mathbb{C}_0, \lD_X)[1] = \bigl(\mathbb{C}[x,x{\textstyle \frac{\partial}{\partial x}}]/(x,1-x{\textstyle\frac{\partial}{\partial x}})\bigr)[-1] = \mathbb{C}_0[-1],
    \]
    where now $x\frac{\partial}{\partial x}$ acts (on the right) as the identity.
    
    In contrast, the dual of the non-holonomic skyscraper $\CC[x\frac{\partial}{\partial x}]$, by a similar computation, is concentrated in cohomological degree $0$.
\end{Ex}

It turns out that this example is typical. 
That is, the failure of $\DVerdier_X$ to send holonomic modules to sheaves is entirely due to shifts concentrated along the log strata.
We can thus correct this failure by using a \enquote{perverse} t-structure on the dual side.

Recall that a t-structure on a triangulated category $\cat{D}$ consists of a pair of full subcategories $(\cat D^{\le 0}, \cat D^{\ge 1})$ subject to the following conditions:
\begin{itemize}
    \item $\cat D^{\le 0}[1] \subseteq \cat D^{\le 0}$ and $\cat D^{\ge 1}[-1] \subseteq \cat D^{\ge 1}$.
    \item $\Hom_{\cat D}(\cat D^{\le 0},\, \cat D^{\ge 1}) = 0$.
    \item Each object $X \in \cat D$ can be embedded in a distinguished triangle
        \begin{equation}\label{eq:t-structure-triangle}
            A \to X \to B
        \end{equation}
        with $A \in \cat D^{\le 0}$ and $B \in \cat D^{\ge 1}$.
\end{itemize}
T-structures were introduced in \cite{BeilinsonBernsteinDeligne:1982:FaisceauxPervers} in order to define perverse sheaves.
The prototypical example is given by the \emph{standard t-structure} of a derived category $\cat D = D(\cat A)$ of an abelian category $\cat A$, where $\cat D^{\le 0}$ (resp.~$\cat D^{\ge 1}$) consists of the complexes whose cohomologies vanish in positive (resp.~non-positive) degrees.

Let us quickly recap the main properties of a t-structure:
For any integer $n$ we set $\cat D^{\le n} = \cat D^{\le 0}[-n]$ and $\cat D^{\ge n} = \cat D^{\ge 1}[-n+1]$.
Then the inclusion $\cat D^{\le n} \hookrightarrow \cat D$ has a right adjoint, often denoted $\tau_{\le n}\colon \cat D \to \cat D^{\le n}$, while $\cat D^{\ge n} \hookrightarrow \cat D$ has a left adjoint, often denoted $\tau_{\ge n}\colon \cat D \to \cat D^{\ge n}$.
In particular, the distinguished triangle~\eqref{eq:t-structure-triangle} is unique up to isomorphism and given by $A = \tau_{\le 0}(X)$ and $B = \tau_{\ge 1}(X)$.
The intersection $\cat D^{\heartsuit} = \cat D^{\le 0} \cap \cat D^{\ge 0}$ is an abelian subcategory, called the \emph{heart} of the t-structure, and the functor
\[
    H^n = \tau_{\le n} \circ \tau_{\ge n} = \tau_{\ge n} \circ \tau_{\le n} \colon \cat D \to \cat D^{\heartsuit}
\]
is a cohomological functor.
In this paper we use the unadorned notation $\tau_{\le n}$ and $\tau_{\ge n}$ for the truncation functors of the standard t-structure, while the functors for other t-structures will have added superscripts.

We come now to the central definition of this paper, the \emph{t-structure of log perverse $\lD_X$-modules}.
Write $r_X = \logdim X - \dim X$ for the generic rank of $\overline{M}_X^{\mathrm{gp}}$ (the reader only interested in non-idealized log varieties may always assume that $r_X = 0$).
Define two full subcategories of the derived category $\catDbqcohDMod{X}$ of quasi-coherent $\lD_X$-modules by
\begin{align*}
    \lsDqcoh^{\le 0}(\lD_X) & = \left\{ \sheaf F \in \catDbqcohDMod{X} : \supp H^{k+r_X}(\sheaf F) \subseteq X^k \text{ for all $k$} \right\}, \\
    \lsDqcoh^{\ge 0}(\lD_X) & = \left\{ \sheaf F \in \catDbqcohDMod{X} : \Gamma_{X^k}(\sheaf F) \in \catDQCoh[\ge k+r_X]{X} \text{ for all $k$}\right\}.
\end{align*}
This t-structure should be thought of as a perverse t-structure with respect to the log stratification.
For example on the standard log line the structure sheaf $\O_X$ is in the heart $\lsDqcoh^{\le 0}(\lD_X) \cap \lsDqcoh^{\ge 0}(\lD_X)$ and so is any skyscraper module at the origin when placed in cohomological degree $1$.
We do however not impose any smoothness condition within the individual strata.
Thus for example the \emph{unshifted} skyscraper D-modules at any point in $\as{1}\setminus\{0\}$ are log perverse.

\begin{Thm}\label{thm:log-perverse-t-structure}
    The pair $(\lsDqcoh^{\le 0}(\lD_X),\, \lsDqcoh^{\ge 0}(\lD_X))$ forms a t-structure on $\catDbqcohDMod{X}$.
\end{Thm}

\begin{proof}
    This is just \cite[Theorem~3.5]{Kashiwara:2004:tStructureOnHolonomicDModuleCoherentOModules} with $\Phi^k = X^k$ and a shift by $r_X$.
\end{proof}

\begin{Def}
    We write
    \[
        \catDMod[\mathrm p]{X} =  \lsDqcoh^{\le 0}(\lD_X) \cap  \lsDqcoh^{\ge 0}(\lD_X) \subseteq  \catDbqcohDMod{X}
    \]
    for the heart of the log perverse t-structure and call its objects \emph{log perverse $\lD_X$-modules}.
    The corresponding truncation functors are denoted by $\tauls_{\le i}$ and $\tauls_{\ge i}$ and the cohomology functors by $\lsH^n = \tauls_{\le n} \circ \tauls_{\ge n}$.
\end{Def}

\begin{Rem}
    The log perverse t-structure does not restrict to a t-structure on $\catDbcohDMod{X}$, i.e.~the pair $\bigl(\lsDqcoh^{\le n}(\lD_X) \cap  \catDbcohDMod{X},\, \lsDqcoh^{\ge n}(\lD_X) \cap \catDbcohDMod{X}\bigr)$ does not form a t-structure.
    For example the sheaf in Remark~\ref{rem:bad_G_filtration} does not have coherent perverse cohomology modules.
\end{Rem}

The definition of $\lsDqcoh^{\ge 0}(\lD_X)$ is additionally motivated by the following proposition, which enhances Proposition~\ref{prop:Xk_local_cohomology_and_support} to a statement about log dimension.
(Of course, as for any t-structure, fixing $\lsDqcoh^{\ge 0}(\lD_X)$ uniquely determines $\lsDqcoh^{>0}(\lD_X)$.)

\begin{Prop}\label{prop:logdim_of_dual_and_ls}
    Let $\sheaf F \in \catDbcohDMod{X}$.
    Then $\Gamma_{X^k}\sheaf F \in \lsDqcoh^{\ge n}(\lD_X)$ if and only if
    \[
        \logdim \res{\Ch\bigl(H^\ell(\DVerdier_X \sheaf F)\bigr)}{X^k} \le \logdim X - n - \ell
    \]
    for all $\ell \in \ZZ$.
\end{Prop}

\begin{proof}
    By definition, $\Gamma_{X^k}\sheaf F \in \lsDqcoh^{\ge n}(\lD_X)$ if and only if $\Gamma_{X^{k+i}}\sheaf F \in \catDQCoh[\ge n + k + i + r_X]{X}$ for all $i \ge 0$.
    By Proposition~\ref{prop:Xk_local_cohomology_and_support} this is the case if and only if
    \begin{align*}
        \logdim \res{\Ch\bigl(H^\ell(\DVerdier_X \sheaf F)\bigr)}{X^k} & =
        \max_{i \ge 0} \dim \res{\Ch\bigl(H^\ell(\DVerdier_X \sheaf F)\bigr)}{X^{k+i}} + k + i + r_X \\ & \le
        \logdim X - n - \ell
        \quad\text{ for all $\ell \in \ZZ$.}
        \qedhere
    \end{align*}
\end{proof}

With the definitions of the log perverse t-structure in place, we can now state the interplay between holonomicity and duality in the logarithmic setting.

\begin{Thm}\label{thm:dual-of-holonomic}
    A coherent $\lD_X$-module $\sheaf F \in \catDModcoh{X}$ is holonomic if and only if $\DVerdier_X(\sheaf F) \in \catDMod[\mathrm p]{X}$.
    Conversely, a coherent log perverse $\lD_X$-module $\sheaf G \in \catDMod[\mathrm p]{X}$ is holonomic if and only if $\DVerdier_X(\sheaf G) \in \catDModcoh{X}$.
\end{Thm}

We split the main arguments of the proof into a series of lemmas.

\begin{Lem}\label{lem:log-defect-and-support}
    Let $\sheaf F$ be a coherent $\lD_X$-module.
    Then, $\supp \sheaf F \subseteq  X^{\dim X - \dim \Ch(\sheaf F)}$.
\end{Lem}

Here we use the convention that $X^k = X$ for $k \le 0$.
The quantity $\dim X - \dim \Ch(\sheaf F)$, if positive, can be seen as measuring how much the classical Bernstein inequality fails.
The lemma then says that the more it fails, the deeper in the log stratification the support of $\sheaf F$ must be located.

\begin{proof}
    Let $U = X \setminus X^{\dim X - \dim \Ch(\sheaf F)}$ and set $\sheaf G = \res{\sheaf F}{U}$.
    If $\sheaf G$ is non-zero, then
    \begin{align*}
        \logdim \Ch(\sheaf G) &< \dim \Ch(\sheaf G) + (\dim X - \dim \Ch(\sheaf G)) + (\logdim X - \dim X) \\& = \logdim X = \logdim U,
    \end{align*}
    in contradiction to Theorem~\ref{thm:log_bernstein}.
\end{proof}

\begin{Lem}\label{lem:D_left_t-exact}
    Let $\sheaf F \in  \catDcohDMod[\ge 0]{X}$.
    Then $\DVerdier_X \sheaf F \in  \lsDqcoh^{\le 0}(\lD_X)$.
\end{Lem}

\begin{proof}
    We have to show that $\supp H^k(\DVerdier_X \sheaf F)$ is contained in $X^k$.
    By Proposition~\ref{prop:Xk_local_cohomology_and_support} we know that
    \[
        \dim \Ch\bigl(H^k(\DVerdier_X \sheaf F)\bigr) \le \logdim X - k \text{  for all $k$}.
    \]
    Thus it follows from Lemma~\ref{lem:log-defect-and-support} that the support of $H^k(\DVerdier_X \sheaf F)$ is contained in $X^{\dim X - (\logdim X - k)} = X^{k-r_X}$.
\end{proof}

\begin{Lem}\label{lem:D_and_logdim_of_ch}
    Let $\sheaf F \in  \catDModcoh{X}$ be a coherent $\lD_X$-module and let $n$ be any integer.
    Then $\logdim \Ch(\sheaf F) \le \logdim X + n$ if and only if $\DVerdier_X\sheaf F \in  \lsDqcoh^{\ge -n}(\lD_X)$.
\end{Lem}

\begin{proof}
    This is just Proposition~\ref{prop:logdim_of_dual_and_ls} applied to $\DVerdier_X \sheaf F$.
\end{proof}

\begin{Lem}\label{lem:perverse_positive_and_dual}
    If $\sheaf F \in \catDbcohDMod{X}$ is contained in $\lsDqcoh^{\ge 0}(\lD_X)$, then $\DVerdier_X\sheaf F \in \catDcohDMod[\le 0]{X}$.
\end{Lem}

\begin{proof}
    By Proposition~\ref{prop:logdim_of_dual_and_ls},
    \[
        \logdim \Ch\bigl(H^\ell(\DVerdier_X \sheaf F)) \le \logdim X - \ell
    \]
    Thus by the logarithmic Bernstein inequality (Theorem~\ref{thm:log_bernstein}) $H^\ell(\DVerdier_X\sheaf F)$ has to vanish for $\ell > 0$.
\end{proof}

\begin{Lem}\label{lem:D_and_logdim_of_ch_2}
    Let $\sheaf F \in \catDMod[\mathrm p]{X}$ be coherent.
    If $\logdim \Ch(\sheaf F) \le \logdim X + n$, then $\DVerdier_X\sheaf F \in  \catDcohDMod[\ge -n]{X}$.
\end{Lem}

\begin{proof}
    By assumption, $\supp H^{\ell+r_X}(\sheaf F)$ is contained in $X^\ell$ for all integers $\ell$.
    Thus $\logdim \Ch(\sheaf F) \le \logdim X + n$ implies
    \[ 
        \dim \Ch\bigl( H^{\ell+r_X}(\sheaf F)\bigr) \le \logdim X - \ell - r_X + n.
    \]
    Hence the result follows from Proposition~\ref{prop:Xk_local_cohomology_and_support}.
\end{proof}

Theorem~\ref{thm:dual-of-holonomic} is now an immediate consequence of Lemmas~\ref{lem:D_left_t-exact} through \ref{lem:D_and_logdim_of_ch_2}.

\begin{Cor}\label{cor:hol_are_p-coherent}
    If $\sheaf F \in \catDbcohDMod{X}$ is holonomic, then so is $\lsH^n(\sheaf F)$ for any $n$.
\end{Cor}

\begin{proof}
    We will show that if $\sheaf G \in \catDbcohDMod{X}$ is holonomic, then $\lsH^n(\DVerdier_X\sheaf G)$ has holonomic cohomology.
    Then the statement follows with $\sheaf G = \DVerdier_X \sheaf F$ as $\DVerdier_X$ preserves holonomicity.

    Using the spectral sequence $\lsH^p(H^q(\sheaf F)) \Rightarrow \lsH^{p+q}$ (see~\cite[Theorem~\RomanNum{IV}.5.1]{KiehlWeissauer:2001:WeilConjecturesPerverseSheavesAndFourierTransform}), we can reduce to $\sheaf G$ being concentrated in a single cohomological degree.
    But then by Theorem~\ref{thm:dual-of-holonomic} $\DVerdier_X\sheaf G$ is concentrated in a single perverse degree and the statement is trivial.
\end{proof}

\section{Two filtrations}\label{sec:filtrations}

Given a coherent $\lD_X$-module $\sheaf F$ one often wants to find the maximal holonomic submodule.
More generally, one is interested in filtering $\sheaf F$ by submodules with increasing log dimension of the associated characteristic varieties.

\begin{Def}
    For a coherent $\lD_X$-module $\sheaf F$ we let $G_i\sheaf F$ denote the largest submodule $\sheaf G$ of $\sheaf F$ with $\logdim\Ch(\sheaf G) \le \logdim X + i$.
    This defines an increasing filtration $G_\bullet$ on $\sheaf F$, called the \emph{log Gabber filtration.}
\end{Def}

We remark that such a largest submodule $G_i\sheaf F$ indeed exists because $\lD_X$ is noetherian and for $\sheaf G, \sheaf H \subseteq  \sheaf F$ we have $\Ch(\sheaf G \cup  \sheaf H) = \Ch(\sheaf G) \cup  \Ch(\sheaf H)$.

\begin{Rem}\label{rem:bad_G_filtration}
    In the classical setting the filtration $G_\bullet$ is quite well behaved.
    In particular $G_0$ never vanishes, i.e.~every coherent D-module has a holonomic submodule.
    In the logarithmic setting this is no longer necessarily the case.

    Consider the example of the standard log plane $X = \as 2$, i.e.~the defining divisor is given by the coordinate axes.
    The $\lD_X$-module $\sheaf F = \lD_X/\lD_X(x-y, x\frac{\partial}{\partial x} + y\frac{\partial}{\partial y})$ is concentrated along the diagonal and has a $2$-dimensional characteristic variety.
    As the fiber of $\Ch(\sheaf F)$ over the origin is $1$-dimensional, $\logdim \Ch(\sheaf F) = 3$ and $\sheaf F$ is not holonomic.
    Consider now $G_0\sheaf F$.
    As $\Ch(G_0\sheaf F)$ is closed and conical, $G_0\sheaf F$ has to vanish in some neighborhood of the origin as otherwise the fiber of $\Ch(G_0\sheaf F)$ over the origin would again be $1$-dimensional.
    Thus $G_0\sheaf F$ has to vanish and $\sheaf F$ has no holonomic submodule.
\end{Rem}

\begin{Rem}\label{rem:no_extension}
    By the same argument the restriction of $\sheaf F$ to $\as 2 \setminus 0$ (which is holonomic) has no holonomic extension to $\as 2$.
    Thus in general if $X \subseteq \bar X$ is a (possibly partial) compactification, one cannot expect that holonomic D-modules on $X$ always extend to holonomic D-modules on $\bar X$.
    This is in contrast with the classical situation, where such an extension is always possible.
    Fundamentally, this is the result of a possibly non-generic interaction of the log structure of $\bar X$ with the characteristic variety of any module.
    
    \label{rem:good_compactification}%
    However we can prove the following local statement: For any normal toric variety $X$ with its toric log structure and any holonomic $\lD_X$-module $\sheaf F$ on $X$, there exists a compact smooth log variety $Y$, containing $X$ as an open subvariety, and a holonomic $\lD_Y$-module $\tilde{\sheaf F}$ such that $\res{\tilde{\sheaf F}}{X} = \sheaf F$.
    This follows by carefully choosing a compactification via toric methods, using Tevelev's Lemma \cite[Lemma 2.2]{tevelev} to control how the support of $\sheaf F$ intersects the new boundary, and an application of Lemma~\ref{lem:ses_by_support}. It will be fleshed out further in upcoming work.
\end{Rem}

In the classical case one obtains good behaviour of the Gabber filtration by comparing it to the Sato--Kashiwara filtration which is defined in an intrinsically functorial way.
In the logarithmic setting, one needs to use log perverse t-structure to define this filtration.
However the fact that this t-structure does not restrict to the coherent subcategory means that the filtration is not always well-defined.
Thus we make the following definition.

\begin{Def}
    A coherent complex $\sheaf F \in \catDbcohDMod{X}$ is called \emph{p-coherent} if each perverse cohomolgy object $\lsH^n(\sheaf F)$ is coherent.
\end{Def}

By Corollary~\ref{cor:hol_are_p-coherent}, every holonomic complex is p-coherent.

\begin{Lem}\label{lem:logdim_of_min_dual}
    If $\sheaf F \in \catDMod[\mathrm p]{X}$ is coherent then
    \[
        \logdim \Ch\bigl(H^{\logdim X - \logdim \Ch(\sheaf F)}(\DVerdier_X \sheaf F)) = \logdim \Ch(\sheaf F).
    \]
\end{Lem}

\begin{proof}
    By Proposition \ref{prop:logdim_of_dual_and_ls}, 
    \[
        \logdim \Ch\bigl(H^{\logdim X - \ell}(\DVerdier_X \sheaf F)) \le \ell
    \]
    and by Lemma~\ref{lem:D_and_logdim_of_ch_2}
    \[
        H^{\logdim X - \ell}(\DVerdier_X \sheaf F) = 0  \text{ for } \ell > \logdim \Ch(\sheaf F).
    \]
    Thus, as $\logdim \Ch(\DVerdier_X \sheaf F) = \logdim \Ch(\sheaf F)$ (Proposition \ref{prop:logdim_of_dual}), necessarily
    \[
        \logdim \Ch\bigl(H^{\logdim X - \logdim\Ch(\sheaf F)}(\DVerdier_X \sheaf F)) = \logdim \sheaf F.
        \qedhere
    \]
\end{proof}

For p-coherent sheaves one has the following dual statement to Proposition~\ref{prop:logdim_of_dual_and_ls}.

\begin{Prop}\label{prop:logdim_of_dual_and_std}
    Let $\sheaf F \in \catDbcohDMod{X}$ and assume that $\DVerdier_X \sheaf F$ is p-coherent.
    If $\sheaf F \in \catDQCoh[\ge n]{X}$ then
    \[
        \logdim \Ch\bigl(\lsH^\ell(\DVerdier_X \sheaf F)\bigr) \le \logdim X - n - \ell
    \]
    for all $\ell \in \ZZ$.
\end{Prop}

\begin{proof}
    For simplicity let us assume that $n=0$.
    The general statement is just a shift of this special case.

    Set $d_\ell = \logdim \Ch\bigl(\lsH^\ell(\DVerdier_X \sheaf F)\bigr) - (\logdim X - \ell)$ and $d^{\max} = \max_\ell d_\ell$.
    If $d^{\max} \le 0$ there is nothing to prove.
    Otherwise let $j$ be any index with $d_{j} = d^{\max}$.
    By Lemma~\ref{lem:D_and_logdim_of_ch_2}, Proposition~\ref{prop:logdim_of_dual_and_ls} and Lemma~\ref{lem:logdim_of_min_dual} we have for all $\ell \in \ZZ$
    \[
        \DVerdier_X\bigl(\lsH^{\ell}(\DVerdier_X\sheaf F)[-\ell]\bigr) \in \catDQCoh[\ge -d_\ell]{X},
    \]
    \[
        \logdim H^{-i}\biggl(\DVerdier_X\bigl(\lsH^{\ell}(\DVerdier_X\sheaf F)[-\ell]\bigr)\biggr) \le  \logdim X - \ell + i \text{ for all $i \in \ZZ$,}
    \]
    and
    \[
        \logdim H^{-d_\ell}\biggl(\DVerdier_X\bigl(\lsH^{\ell}(\DVerdier_X\sheaf F)[-\ell]\bigr)\biggr) = \logdim X - \ell + d_\ell.
    \]
    Applying duality to the triangle
    \[
        \lsH^j(\DVerdier_X\sheaf F)[-j]
        \to
        \tauls_{\ge j} \DVerdier_X \sheaf F
        \to
        \tauls_{\ge j+1}\DVerdier_X \sheaf F
    \]
    gives the exact sequence
    \begin{align*}
        0
        \to
        H^{-d^{\max}}\bigl(\DVerdier_X \tauls_{\ge j} \DVerdier_X \sheaf F\bigr) 
        & \to
        H^{-d^{\max}}\bigl(\DVerdier_X \bigl(\lsH^j(\DVerdier_X\sheaf F)[-j] \bigr)\bigr)
        \\ & \to
        H^{-d^{\max}+1}\bigl(\DVerdier_X \tauls_{\ge j+1} \DVerdier_X \sheaf F\bigr).
    \end{align*}
    By the discussion above,
    \begin{align*}
        \logdim H^{-d^{\max}+1}\bigl(\DVerdier_X \tauls_{\ge j+1} \DVerdier_X \sheaf F\bigr)
        & \le
        \logdim X - (j+1) + d_{\max}
        \\&\lneq
        \logdim H^{-d^{\max}}\bigl(\DVerdier_X \bigl(\lsH^j(\DVerdier_X\sheaf F)[-j] \bigr)\bigr).
    \end{align*}
    Thus the last morphism cannot be injective and hence $H^{-d^{\max}}(\DVerdier_X\tauls_{\le j}\DVerdier_X\sheaf F) \ne 0$.
    Conversely, $\DVerdier_X\tauls_{< j}\DVerdier_X\sheaf F \in \catDQCoh[\ge -d^{\max}]{X}$.
    The long exact sequence coming from the triangle
    \[
        \DVerdier_X\tauls_{\ge j}\DVerdier_X\sheaf F \to \DVerdier_X\DVerdier_X\sheaf F \to 
        \DVerdier_X\tauls_{< j}\DVerdier_X\sheaf F
    \]
    thus contains the sequence
    \[
        0 \to H^{-d^{\max}}(\DVerdier_X\tauls_{\ge j}\DVerdier_X\sheaf F)
        \to H^{-d^{\max}}(\sheaf F).
    \]
    As $d^{\max } > 0$ the term $H^{-d^{\max}}(\sheaf F)$ vanishes by assumption, yielding a contradiction.
\end{proof}

\begin{DefLem}
    For a coherent $\lD_X$-module $\sheaf F$ such that $\DVerdier_X\sheaf F$ is p-coherent let $S_i\sheaf F$ be the image of the canonical morphism
    \[
        H^0(\DVerdier_X \tauls_{\ge -i} \DVerdier_X \sheaf F) \to  H^0(\DVerdier_X \circ  \DVerdier_X \sheaf F) = \sheaf F.
    \]
    Then $S$ is an increasing filtration on $\sheaf F$, which we call the \emph{log Sato-Kashiwara filtration}.
\end{DefLem}

Before we prove that this is actually a filtration, let us state our reason for introducing the log Sato--Kashiwara filtration, i.e.~that it coincides with the log Gabber filtration.

\begin{Thm}\label{thm:G=S}
    Let $\sheaf F$ be a coherent $\lD_X$-module and assume that $\DVerdier_X\sheaf F$ is p-coherent.
    Then log Gabber and log Sato-Kashiwara filtrations of $\sheaf F$ agree, i.e.
    \[
        G_i\sheaf F = S_i\sheaf F \quad \text{for all $i \in  \mathbb{Z}$}.
    \]
\end{Thm}

\begin{Lem}\label{lem:dual_lsH_dual_vanishing}
    Suppose $\sheaf F \in  \catDModcoh{X}$ is a coherent $\lD_X$-module such that $\DVerdier_X$ is p-coherent.
    Then 
    \[
        \DVerdier_X\bigl(\lsH^\ell(\DVerdier_X\sheaf F)\bigr) \in  \catDcohDMod[\ge \ell]{X}
    \]
    for all $\ell \in  \mathbb{Z}$.
\end{Lem}

\begin{proof}
    Follows from Proposition~\ref{prop:logdim_of_dual_and_std} and Lemma~\ref{lem:D_and_logdim_of_ch}.
\end{proof}

\begin{proof}[Proof that $S_i$ is an increasing filtration]
    To start, let us remark that since $\DVerdier_X\sheaf F$ is p-coherent and hence $\tauls_{\ge -i}$ preserves coherence, $S_i\sheaf F$ is indeed well-defined.
    Let us note that by Lemma~\ref{lem:D_and_logdim_of_ch}, $S_{\logdim \Ch(\sheaf F)}\sheaf F = \sheaf F$.
    We have to show that $S_i\sheaf F \subseteq  S_{i+1}\sheaf F$ for all $i$.
    Consider the distinguished triangle
    \[
        \lsH^{-(i+1)}(\DVerdier \sheaf F)[i+1] \to  
        \tauls_{\ge -(i+1)}(\DVerdier \sheaf F) \to  
        \tauls_{\ge -i}(\DVerdier \sheaf F),
    \]
    apply duality and consider the corresponding long exact sequence on cohomology
    \begin{align*}
        H^{-(i+1)-1}\bigl(\DVerdier_X(\lsH^{-(i+1)}(\DVerdier \sheaf F))\bigr) &\to
        H^0\bigl(\DVerdier_X\tauls_{\ge -i}(\DVerdier \sheaf F)\bigr) \\ &\to
        H^0\bigl(\DVerdier_X\tauls_{\ge -(i+1)}(\DVerdier \sheaf F)\bigr).
    \end{align*}
    The first term vanishes by Lemma~\ref{lem:dual_lsH_dual_vanishing}.
    Thus we see that $S_i\sheaf F \subseteq  S_{i+1}\sheaf F$.
\end{proof}

We point out that from the proof it follows in particular that the morphism $H^0(\DVerdier_X \tauls_{\ge -i} \DVerdier_X \sheaf F) \to  \sheaf F$ defining the log Sato--Kashiwara filtration is injective.

\begin{proof}[Proof of Theorem~\ref{thm:G=S}]
    Fix a coherent $\lD_X$-module $\sheaf F$.
    By Proposition~\ref{prop:logdim_of_dual_and_ls} we have
    \[
        \logdim \Ch(S_i\sheaf F) = 
        \logdim \Ch\bigl(H^0(\DVerdier_X\tauls_{\ge -i} \DVerdier_X \sheaf F)\bigr)
        \le \logdim X + i,
    \]
    or in other words $S_i\sheaf F \subseteq G_i\sheaf F$.

    Let us now show that $G_i\sheaf F \subseteq  S_i\sheaf F$ by induction on $i$.
    By the log Bernstein inequality and Lemma~\ref{lem:D_left_t-exact} respectively, both filtrations vanish for $i < 0$, providing the base case.
    Assume now that $G_{i-1}\sheaf F \subseteq S_{i-1}\sheaf F$.
    If $G_{i-1}\sheaf F = G_i\sheaf F$, then by induction 
    \[
        G_i\sheaf F = G_{i-1}\sheaf F \subseteq  S_{i-1}\sheaf F \subseteq  S_i \sheaf F.
    \]
    So assume that $G_{i-1}\sheaf F \subsetneq G_i\sheaf F$.
    The definition of $S_i$ is functorial, so that the inclusion $G_i\sheaf F \hookrightarrow \sheaf F$ gives 
    \[
        S_i(G_i\sheaf F) \hookrightarrow S_i(\sheaf F)
    \]
    and it suffices to show that $S_i(G_i\sheaf F) = G_i \sheaf F$.
    As $G_{i-1}\sheaf F \subsetneq G_i\sheaf F$ we see that $\logdim \Ch(G_i\sheaf F)$ has to be equal to $\logdim X + i$.
    Thus the result follows from $S_{\logdim \Ch(\sheaf G) - \logdim X}\sheaf G = \sheaf G$ for any coherent $\sheaf G$ (cf.~Lemma~\ref{lem:D_and_logdim_of_ch}).
\end{proof}

\iffalse
\begin{Cor}\label{cor:holonomics-can-be-extended}
    Let $j\colon U \hookrightarrow X$ be a strict open immersion, and $\sheaf F$ be a holonomic $\lD_U$-module.
    Then there exists a holonomic $\lD_X$-module $\sheaf F'$ such that $\res{\sheaf F'}{U}= \sheaf F$.
\end{Cor}

\begin{proof}
    Using Lemma~\ref{lem:coherent_subs}, choose a coherent $\lD_X$-module $\widetilde{\sheaf F}$ such that $\res{\widetilde{\sheaf F}}{U} = \sheaf F$.
    Set $\sheaf F' = G_0\widetilde{\sheaf F}$.
    By construction $\sheaf F'$ is holonomic and from the functoriality of the log Sato-Kashiwara filtration it follows that
    \[
        \res{\sheaf F'}{U} = G_0\bigl(\res{\widetilde{\sheaf F}}{U}\bigr) = G_0\sheaf F = \sheaf F.
        \qedhere
    \]
\end{proof}

\begin{Rem}
  The extension constructed of Corollary~\ref{cor:holonomics-can-be-extended} ultimately depends on the choice of local generators for $\sheaf F$.
  Different choices will lead to different coherent extensions $\widetilde{\sheaf F}$ in the proof and hence $\sheaf F'$ is not canonically defined.
  In future work we will discuss how to use a generalization of the Malgrange--Kashiwara V-filtration to define canonical extensions (up to a choice of a section of $\CC \to \CC/\ZZ$, as usual).
\end{Rem}

\fi

\printbibliography

\end{document}

%% file: macros.tex
\newcommand\kk{\Bbbk}               % general base field
\newcommand\CC{\mathbb{C}}          % complex numbers
\newcommand\NN{\mathbb{N}}          % natural numbers
\newcommand\ZZ{\mathbb{Z}}          % integers
\newcommand\LL{{\mathbb{L}}}        % left derived
\newcommand\RR{{\mathbb{R}}}        % right derived
\newcommand\Lotimes{\otimes^\LL}    % derived tensor product
\newcommand\from{\leftarrow}
\DeclareMathOperator\Hom{Hom}
\DeclareMathOperator\Ext{Ext}
\newcommand\cat\mathbf
\newcommand\ideal\mathfrak
\newcommand\res[2]{\mathchoice{\left.#1\right|_{#2}}{#1|_{#2}}{#1|_{#2}}{#1|_{#2}}} % restriction
\newcommand\rquot[2]{%              quotient #1/#2
    \mathchoice%
        {\left.#1\kern-0.2ex\middle/\kern-0.3ex\lower0.7ex\hbox{$\displaystyle #2$}\right.}%
        {\left.#1\middle/#2\right.}%
        {\left.#1\middle/#2\right.}%
        {\left.#1\middle/#2\right.}%
}

\newcommand\as[2][]{\mathbb A^{#2}_{#1}}        % affine space
\DeclareMathOperator\Spec{Spec}     % spectrum of a ring

\newcommand\sheaf\mathcal
\renewcommand\O{\sheaf O}
\newcommand\sheafHom{\underline{\Hom}} % sheafy Hom
\newcommand\sheafExt{\underline{\Ext}} % sheafy Ext
\newcommand\DD{\mathbb{D}}          % duality
\newcommand\canon[1][X]{\omega_{#1}}    % (log) canonical sheaf
\newcommand\DC{\sheaf R}            % dualizing complex
\newcommand\catCoh[1]{\cat{Coh}(#1)}                        % category of coherent sheaves
\newcommand\catDCoh[2][]{\cat D^{#1}_{\mathrm{coh}}(#2)}    % derived category of coherent sheaves
\newcommand\catDbCoh[1]{\catDCoh[b]{#1}}                    % bounded derived category of coherent sheaves
\newcommand\catDQCoh[2][]{\cat D^{#1}_{\mathrm{qc}}(#2)}    % derived category of quasi-coherent sheaves
\newcommand\catDPlusCoh[1]{\cat D^+\catCoh{#1}}
\DeclareMathOperator\supp{supp}     % support
\DeclareMathOperator\gldim{gldim}   % global dimension

\newcommand\logdim{\operatorname{log\;dim}}
\newcommand\DVerdier{\DD^{\mathrm{Ve}}}    % Verdier duality
\newcommand\lD{\sheaf{D}}           % (log) differential operators

\newcommand\catDMod[2][]{\cat{Mod}_{#1}(\lD_{#2})}
\newcommand\catDDMod[2][]{\cat{D}^{#1}(\lD_{#2})}               % derived category of D-modules
\newcommand\catDbDMod[2][]{\cat{D}^b_{#1}(\lD_{#2})}            % bounded derived category of D-modules
\newcommand\catDModcoh[1]{\catDMod[\mathrm{coh}]{#1}}           % abelian category of coherent D-modules
\newcommand\catDcohDMod[2][]{\cat{D}_{\mathrm{coh}}^{#1}(\lD_{#2})} % derived category of coherent D-modules
\newcommand\catDbcohDMod[1]{\catDcohDMod[b]{#1}}                % bounded derived category of coherent D-modules
\newcommand\catDbqcohDMod[1]{\cat{D}_{\mathrm{qc}}^b(\lD_{#1})} % derived category of quasi-coherent D-modules

\newcommand\Ch{\mathrm{Ch}}         % characteristic variety

\newcommand\epb{{*}}                % "easy"-pushforward
\newcommand\spf{{\bullet}}          % "star"-pushforward